\documentclass[11pt]{article}

\usepackage{geometry}
\usepackage{amssymb}
\usepackage{amsthm}
\usepackage{amsmath}
\numberwithin{equation}{section}
\numberwithin{figure}{section}
\usepackage{mathrsfs}
\usepackage{graphicx, caption}
\usepackage{color}
\usepackage{appendix}
\usepackage{enumerate}

\theoremstyle{plain}

\newtheorem{theorem}{Theorem}[section]
\newtheorem{lemma}[theorem]{Lemma}
\newtheorem{proposition}[theorem]{Proposition}
\newtheorem{corollary}[theorem]{Corollary}
\newtheorem{definition}{Definition}[section]
\newtheorem{remark}{Remark}[section]
\theoremstyle{remark}
\newtheorem{claim}{Claim}

\newtheorem{question}{Question}

\def\xX{\mathscr{X}}

\def\xC{\mathscr{C}}

\def\xM{\mathscr{M}}
\def\xP{\mathscr{P}}

\def\cC{{\mathcal C}}
\def\cG{{\mathcal G}}
\def\cH{{\mathcal H}}

\def\cN{{\mathcal N}}
\def\cU{{\mathcal U}}
\def\cV{{\mathcal V}}

\def\Cl{{\rm Cl}}
\def\e{{\rm e}}
\def\vep{\varepsilon}

\def\si{\sigma}
\def\Si{\Sigma}
\def\La{\Lambda}

\def\Sing{\operatorname{Sing}}
\def\Per{\operatorname{Per}}

\def\Orb{\operatorname{Orb}}

\title{Generic properties of vector fields identical on a compact set and codimension one partially hyperbolic dynamics}

\author{Shaobo Gan, Ruibin Xi, Jiagang Yang, Rusong Zheng\thanks{S. Gan is partially supported by National Key R\&D Program of China 2022YFA1005801 and NSFC 12161141002. J. Yang is partially supported by CNPq, FAPERJ, PRONEX, MATH-AmSud 220029, and NSFC grants 12271538, 11871487 and 12071202. R. Zheng is partially supported by NSFC grants 12071007, 12101293.}}

\begin{document}

\maketitle

\begin{abstract}
Let $\xX^r(M)$ be the set of $C^r$ vector fields on a boundaryless compact Riemannian manifold $M$.  Given a vector field $X_0\in\xX^r(M)$ and a compact invariant set $\Gamma$ of $X_0$, we consider the closed subset $\xX^r(M,\Gamma)$ of $\xX^r(M)$, consisting of all $C^r$ vector fields which coincide with $X_0$ on $\Gamma$. Study of such a set naturally arises when one needs to perturb a system while keeping part of the dynamics untouched.

A vector field $X\in\xX^r(M,\Gamma)$ is called $\Gamma$-avoiding Kupka-Smale, if the dynamics away from $\Gamma$ is Kupka-Smale.
We show that a generic vector field in $\xX^r(M,\Gamma)$ is $\Gamma$-avoiding Kupka-Smale.
In the $C^1$ topology, we obtain more generic properties for  $\xX^1(M,\Gamma)$. With these results, we further study  codimension one partially hyperbolic dynamics for generic vector fields in $\xX^1(M,\Gamma)$, giving a dichotomy of hyperbolicity and Newhouse phenomenon.

As an application, we obtain that $C^1$ generically in $\xX^1(M)$, a non-trivial Lyapunov stable chain recurrence class of a singularity which admits a codimension 2 partially hyperbolic splitting with respect to the tangent flow is a homoclinic class.
\end{abstract}


\section{Introduction}
Let $M$ be a boundaryless compact Riemannian manifold. We denote by $\xX^r(M)$ the set of all $C^r$ vector field on $M$. Let $X_0\in\xX^r(M)$ and $\Gamma\subset M$ be any $X_0$-invariant compact subset, i.e. $\phi^{X_0}_t(\Gamma)=\Gamma$ for all $t$, where $\phi^{X_0}_t$ is the flow generated by $X_0$.
Let us denote
\[\xX^r(M,\Gamma)=\{X\in\xX^r(M): X|_{\Gamma}=X_0|_{\Gamma}\},\]
which is a nonempty closed subset of $\xX^r(M)$.
Study of such a set naturally arises when one needs to perturb a system while keeping part of the dynamics untouched.
This is often a subtle subject if the dynamics to be kept untouched is not persistent per se, e.g. nonhyperbolic periodic orbits, heterodimensional cycles, and homoclinic orbits of singularities.

An early example appeared in Palis' proof of the $C^1$ $\Omega$-stability conjecture \cite{Pa}, where he claimed that a diffeomorphism with a heterodimensional cycle $\Gamma$ can be perturbed such that no other heterodimensional cycles exist except this one. This can be shown by proving a generic property (Theorem \ref{thm.Gamma-KS}) of the set $\xX^r(M,\Gamma)$. However, a rigorous proof was missing.
Recently, in the study differentiable flows,  Gan-Yang \cite[Proposition 6.23]{GY} applied also a similar result allowing one to perturb a vector field to be weak Kupka-Smale while preserving some homoclinic orbits of singularities.   Shi-Gan-Wen \cite[Sublemma 4.6]{SGW} considered generic properties of a set $\xX^1(M,\Gamma)$ where $\Gamma$ is the closure of a heteroclinic orbit connecting two critical elements (singularities or periodic orbits).
These examples justify the study of generic properties of the set $\xX^r(M,\Gamma)$.

\subsection{Generic properties of $\xX^r(M,\Gamma)$}
Note that $\xX^r(M,\Gamma)$ is a Baire space, endowed with the relative topology as a closed subset of $\xX^r(M)$. Thus it makes sense to consider generic properties of the space $\xX^r(M,\Gamma)$, i.e. properties that hold for a residual subset of $\xX^r(M,\Gamma)$. In fact, the problems leading to consideration of a set as $\xX^r(M,\Gamma)$ can often be solved by proving a generic property of $\xX^r(M,\Gamma)$, as is the case for Palis \cite{Pa}.

The Kupka-Smale property is a fundamental generic property of differentiable dynamical systems, see \cite{PdM}, for instance.
We define a variation of Kupka-Smale property for vector fields in $\xX^r(M,\Gamma)$.
\begin{definition}\label{def.kupka-smale}
  A vector field $X\in\xX^r(M,\Gamma)$ is called {\em $\Gamma$-avoiding Kupka-Smale} if the following conditions hold:
  \begin{itemize}
    \item every critical element (a periodic orbit or singularity) of $X$ outside of $\Gamma$ is hyperbolic;
    \item for any critical elements $c_1$ and $c_2$ of $X$, outside of $\Gamma$, the stable manifold $W^s(c_1,X)$ of $c_1$ is transverse to the unstable manifold $W^u(c_2,X)$ of $c_2$.
  \end{itemize}
\end{definition}
The $\Gamma$-avoiding Kupka-Smale property is fundamental in the study of generic properties of $\xX^r(M,\Gamma)$, as will be seen in Section \ref{sect.genericity}. In particular, we mention that it allows the use of connecting lemma for chains away from $\Gamma$, see Lemma \ref{lem.connecting}.
Denote by $KS(M,\Gamma)$ the set of $\Gamma$-avoiding Kupka-Smale vector fields in $\xX^r(M,\Gamma)$. We have the following result. 
\begin{theorem}
\label{thm.Gamma-KS}
  $KS(M,\Gamma)$ is a residual subset of $\xX^r(M,\Gamma)$.
\end{theorem}

Be aware that generic properties of $\xX^r(M)$ may not hold for any vector field in $\xX^r(M,\Gamma)$. This is because $\xX^r(M,\Gamma)$ is usually a meagre subset of $\xX^r(M)$.
However, as in this theorem, one does expect that generic properties of $\xX^r(M)$ holds also for $\xX^r(M,\Gamma)$, after appropriate change of conditions.

We now focus on the $C^1$ topology. Many of the generic properties of $\xX^1(M)$ are related to periodic orbits and chain transitive sets. 
For example, a well-known generic property of $\xX^1(M)$ states that a non-trivial compact chain-transitive set is approximated by periodic orbits \cite{Cr06}.
We show that this property also holds for generic vector fields in $\xX^1(M,\Gamma)$, in the following sense:
\begin{theorem}\label{thm.genericity}
  There is a residual subset $\cG_1\subset \xX^1(M,\Gamma)$ such that for each $X\in \cG_1$, if $\Lambda$ is a non-trivial compact chain-transitive set and $\Lambda\cap\Gamma=\emptyset$, then $\Lambda$ is the Hausdorff limit of a sequence of periodic orbits of $X$.
\end{theorem} 
Generalizing a $C^1$ generic property obtained in \cite{Wen04}, we can improve Theorem \ref{thm.genericity} as follows.
We say that a compact invariant set $\Lambda$ of $X\in\xX^1(M,\Gamma)$ is an {\em index $i$ fundamental limit within $\xX^1(M,\Gamma)$}, if there exists a sequence of vector fields $X_n\in\xX^1(M,\Gamma)$ converging to $X$ in the $C^1$ topology and each $X_n$ has an index $i$ periodic orbit $\gamma_n$ such that $\gamma_n\to \Lambda$ in the Hausdorff topology. Recall that the index of a hyperbolic periodic orbit $\gamma$ is $\dim W^s(\gamma)-1$, where $W^s(\gamma)$ is its stable manifold.
\begin{theorem}\label{thm.f-limit}
  There is a residual subset $\cG_2\subset \xX^1(M,\Gamma)$ such that for each $X\in \cG_2$, let $\Lambda$ be a compact invariant set, if $\Lambda\cap\Gamma=\emptyset$ and $\Lambda$ is an index $i$ fundamental limit within $\xX^1(M,\Gamma)$, then $\Lambda$ is the Hausdorff limit of a sequence of index $i$ periodic orbits of $X$.
\end{theorem}

Another classical generic property of $\xX^1(M)$ states that a chain recurrence class containing a periodic orbit coincides with the homoclinic class of the periodic orbit \cite{Cr06}. This property can be localized \cite[Theorem 1.1]{Cr10}: $C^1$ generically, any chain-transitive set $K$ containing a periodic point $P$ is contained in the relative homoclinic class of $P$ in an arbitrarily small neighborhood $U$ of $K$. We generalized this result to $\xX^1(M,\Gamma)$.

\begin{theorem}\label{thm.relative-h-class}
  There is a residual subset $\cG_3\subset\xX^1(M,\Gamma)$ such that for each $X\in\cG_3$, let $\Lambda$ be a compact chain transitive set, if $\Lambda\cap\Gamma=\emptyset$ and $\Lambda$ contain a periodic orbit $\gamma$, then for any neighborhood $U$ of $\Lambda$ one has $\Lambda\subset H(\gamma,U)$. Here, $H(\gamma,U)$ is the relative homoclinic class of $\gamma$ in $U$.
\end{theorem}

We shall see from the proof of these theorems that a bulk of generic properties of $\xX^r(M)$ can be generalized to $\xX^r(M,\Gamma)$ almost immediately. A natural question goes as follows.
\begin{question}
  Suppose that a property $\xP$ is stated for all critical elements (or compact chain transitive sets, homoclinic classes, etc.), and suppose $\xP$ holds generically in $\xX^r(M)$. Is it true that the same property $\xP$ holds also generically in $\xX^r(M,\Gamma)$, as long as one considers only the critical elements (compact chain transitive sets, homoclinic classes, etc.) outside of $\Gamma$?
\end{question}
Sadly, the answer to this question is in the negative. The reason is that the property $\xP$ for a specific object may get involved with dynamics inside $\Gamma$, although the object itself is away from $\Gamma$. For example, Theorem \ref{thm.relative-h-class} cannot be stated as the following, which is false in general: for a generic vector field $X\in\xX^1(M,\Gamma)$, for any periodic orbit $\gamma$ of $X$, if $\gamma\cap\Gamma=\emptyset$, then the chain recurrence class containing $\gamma$ coincides with the homoclinic class of $\gamma$.

In the spirit of Question 1, we have the following principle: {\em suppose that a property $\xP$ is stated for all critical elements (or compact chain transitive sets, homoclinic classes, etc.) and $\xP$ holds generically in $\xX^r(M)$, then the property $\xP$ holds also generically in $\xX^r(M,\Gamma)$, as long as one considers $\xP$ only for the critical elements (compact chain transitive sets, homoclinic classes, etc.) such that the involved dynamics is away from $\Gamma$.}
For instance, it can be show that for a generic vector field $X\in\xX^1(M,\Gamma)$, any homoclinic class $\Lambda$ away from $\Gamma$ is index-complete: if $\Lambda$ contains periodic orbits of both index $i$ and index $j$, $i<j$, then it contains periodic orbits of index $k$ for every integer $k\in[i,j]$. This corresponds to a classical generic property in $\xX^1(M)$ obtained in \cite{ABCDW}.

\subsection{Application to codimension one partially hyperbolic dynamics}

In this paper, we shall content ourselves with the generic results stated in the theorems above and apply them to codimension one partially hyperbolic dynamics.
Our goal is the following theorem, which generalizes a key result of Gan-Yang \cite[Theorem 4.4]{GY} from dimension 3 to higher dimensions\footnote{It is possible to prove a stronger result: under the assumptions of Theorem \ref{thm.h-class}, $C(\sigma)$ is singular hyperbolic. By the theorem, it remains to show that $C(\sigma)$ is singular hyperbolic {\em with the knowledge that it contains a periodic orbit}. This has been done by Bonatti-Gan-Yang \cite{BGY} for three dimensional flows.}.
\begin{theorem}\label{thm.h-class}
  Let $X\in\xX^1(M)$ be a $C^1$ generic vector field. Suppose $C(\sigma)$ is a non-trivial chain recurrence class of a singularity $\sigma$ satisfying the following conditions:
  \begin{itemize}
    \item all singularities in $C(\sigma)$ are Lorenz-like;
    \item there is a partially hyperbolic splitting $T_{C(\sigma)}M=E^{s}\oplus F$ with respect to the tangent flow, where $E^{s}$ is contracting and $\dim F=2$.
  \end{itemize}
  Then $C(\sigma)$ is a homoclinic class.
\end{theorem}
Here, Lorenz-like singularities are defined as follows.
\begin{definition}[Lorenz-like singularity]\label{def:lorenz-like}
	Let $X\in\xX^1(M)$ and $\si$ be a hyperbolic singularity of $X$. Assume that $C(\si)$ is non-trivial and the Lyapunov exponents of $D\phi_t|_{T_{\si}M}$ are $\lambda_1\leq\lambda_2\leq\cdots\leq\lambda_s<0<\lambda_{s+1} \leq\cdots\leq\lambda_{\dim M}$.
	The singularity $\si$ is said to be {\em Lorenz-like}, if the following conditions are satisfied:
	\begin{itemize}
		\item $\lambda_s+\lambda_{s+1}>0$, $s>1$ and $\lambda_{s-1}<\lambda_s$;
		\item Let $E^{ss}$ be the invariant subspace corresponding to $\lambda_1,\cdots,\lambda_{s-1}$, and $W^{ss}(\si)$  the strong stable manifold corresponding to $E^{ss}$, then $W^{ss}(\si)\cap C(\si)=\{\si\}$.
	\end{itemize}
\end{definition}

\begin{remark}\label{rmk.h-class}
  In the setting of Theorem \ref{thm.h-class}, one has that $C(\sigma)$ is  Lyapunov stable: the codimension 2 partially hyperbolic splitting implies that the unstable manifold of $\sigma$ is one-dimensional and at least one separatrix of the unstable manifold is contained in $C(\sigma)$, then $C(\sigma)$ is Lyapunov stable by genericity \cite[Appendix A]{APa}.

  It is shown by Pacifico, F. Yang and J. Yang \cite{PYY} that for a $C^1$ generic vector field, if a non-trivial chain recurrence class is Lyapunov stable and singular hyperbolic, then it is a homoclinic class. See also \cite{ALM}.
\end{remark}

For the proof of Theorem \ref{thm.h-class}, which follows the argument in \cite[Section 6]{GY}, we will take $\Gamma$ to be a union of homoclinic orbits of singularities and consider generic properties in $\xX^1(M,\Gamma)$. Precisely, we have the following result,  which is a generalization of a main result (Lemma 5.1) in \cite{Y11} to flows.

\begin{theorem}\label{thm.codim-one}
  There exists a residual subset $\cG\subset\xX^1(M,\Gamma)$ such that for every $X\in\cG$ and any compact chain-transitive set $\Lambda$, if $\Lambda\cap(\Gamma\cup\Sing(X))=\emptyset$ and its normal bundle admits a partially hyperbolic splitting $\cN_{\Lambda}=N^s\oplus N^c$ with $N^s$ contracting and $\dim N^c=1$, then either $\Lambda$ is a hyperbolic set, or it is not hyperbolic and is contained in a chain-transitive set $\Lambda'$ which contains a periodic saddle of index $\dim M-2$.

  In the latter case, there are periodic sinks contained in the closure of the unstable set of $\Lambda'$ and arbitrarily close to $\Lambda'$. Furthermore, if $\Lambda$ is minimally nonhyperbolic, it is the Hausdorff limit of periodic sinks.
\end{theorem}

\begin{remark}
  In the latter case, the chain-transitive set $\Lambda'$ is contained in a relative homoclinic class in a small neighborhood of $\Lambda$. This is given by Theorem \ref{thm.relative-h-class}.
\end{remark}

\begin{remark}\label{rmk.codim-one-1}
  The theorem allows us to bypass a $C^2$ argument in the original proof of Gan-Yang \cite[Section 6]{GY}.
  More precisely, they use the following result in \cite{ARH,PS00}: suppose $\dim M=3$ and let $X\in\xX^2(M)$, if $\Lambda$ is a compact invariant transitive set of $X$ such that
  \begin{itemize}
    \item $\Lambda$ contains no singularities;
    \item the normal bundle $\cN_{\Lambda}$ admits a dominated splitting with respect to the linear Poincar\'e flow;
    \item every periodic orbit in $\Lambda$ is hyperbolic, but $\Lambda$ is not hyperbolic;
  \end{itemize}
  then $\Lambda$ is a normally hyperbolic 2-dimensional torus, restricted to which the flow is equivalent to an irrational flow.

  A generalization of this result to codimension one partially hyperbolic dynamics on higher dimensional manifolds is long-desired, yet a rigorous proof is not available. We would like to mention that for diffeomorphisms, important progresses have been obtained by Pujals-Sambarino \cite{PS07,PS09}. For instance, they proved in \cite{PS09} that if a homoclinic class $H(p)$ of a $C^2$ diffeomorphism admits a codimension one partially hyperbolic splitting and periodic orbits in the class are all hyperbolic, then $H(p)$ is hyperbolic.
\end{remark}

Theorem \ref{thm.codim-one} gives a dichotomy of hyperbolicity and Newhouse phenomenon \cite{Ne70,Ne74} for dynamics admitting a codimension-one partially hyperbolic splitting.
Note that if $\Gamma=\emptyset$, we obtain a generic result in $\xX^1(M)$.

The rest of the paper is organized as follows:
we study generic properties of $\xX^r(M,\Gamma)$ in Section \ref{sect.genericity} and prove Theorem \ref{thm.Gamma-KS} $\sim$ \ref{thm.relative-h-class}; then in Section \ref{sect.codim-one}, we apply these results to codimension one partially hyperbolic dynamics and prove Theorem \ref{thm.codim-one}; finally, we prove Theorem \ref{thm.h-class} in Section \ref{sect.pf-of-h-class}.

\section{Generic properties of $\xX^r(M,\Gamma)$}\label{sect.genericity}
We introduce in this section the fundamentals of flows and prove the $\Gamma$-avoiding Kupka-Smale property and other generic properties of $\xX^r(M,\Gamma)$ stated in the first section.

Let $X\in\xX^r(M)$. The vector field generates a $C^r$ flow on $M$, which will be denoted by $\phi^X_t$, or simply, $\phi_t$.
The flow $\phi_t$ induces a tangent flow on the tangent bundle $TM$, denoted by $\Phi_t=D\phi_t$.
A point $x\in M$ is called a {\em singularity} if $X(x)=0$. Otherwise, it is called {\em regular}.
We denote by $\Sing(X)=\{x\in M: X(x)=0\}$ the set of singularities of $X$.
If $x$ is regular and there exists $t>0$ such that $\phi_t(x)=x$, then $x$ is called a {\rm periodic point} and its {\em period} is defined to be the minimum positive $t$ such that $\phi_t(x)=x$. The set of periodic points of $X$ is denoted by $\Per(X)$.
A {\em critical element} of $X$ is either a singularity or a periodic orbit of $X$.

For a singularity $\sigma$, one says that $\sigma$ is hyperbolic if $\Phi_1|_{T_{\sigma}M}$ is nondegenerate and has no eigenvalue on the unit circle. Its index is defined to be the number of eigenvalues (counting multiplicity) contained in the unit disk.
For a periodic point $p$, we denote by $\tau(p)$ its period. The periodic point $p$, or its orbit $\Orb(p)$, is called {\em hyperbolic} if $\Phi_{\tau(p)}|_{T_pM}$ has no eigenvalue on the unit circle except the one corresponding to the flow direction.

As in the introduction, we shall always denote by $\Gamma$ as a $\phi^{X_0}_t$-invariant compact subset of $M$, where $X_0\in\xX^r(M)$ is a given vector field.
Let us denote
\[\xX^r(M,\Gamma)=\{X\in\xX^r(M): X|_{\Gamma}=X_0|_{\Gamma}\},\]

\subsection{$\Gamma$-avoiding Kupka-Smale property}

This subsection provides a proof of Theorem \ref{thm.Gamma-KS}, i.e. $\Gamma$-avoiding Kupka-Smale property is generic in $\xX^r(M,\Gamma)$.
Before the proof, we would like to mention the second author's master thesis \cite{Xi}, in which Theorem \ref{thm.Gamma-KS} first appeared. Xi provided a sketch of the proof after proving its diffeomorphism version. Our proof will follow closely this sketch.
For simplicity, we will regard singularities as critical elements of period 0.

For any $T>0$ and $\delta>0$, let $\xX^r_T(M,\Gamma,\delta)$ be the set of vector fields in $\xX^r(M,\Gamma)$ such that every critical element lying outside of $B_{\delta}(\Gamma)$ and having period $\leq T$ is hyperbolic.
Here, $B_{\delta}(\Gamma)$ denotes the open $\delta$-neighborhood of $\Gamma$, and a critical element is said to be lying outside of $B_{\delta}(\Gamma)$ if it does not intersect $B_{\delta}(\Gamma)$.

\begin{lemma}\label{lem.h-critical-elements}
  $\xX^r_T(M,\Gamma,\delta)$ is open and dense in $\xX^r(M,\Gamma)$.
\end{lemma}
\begin{proof}
  Let us show firstly that $\xX^r_T(M,\Gamma,\delta)$ is open in $\xX^r(M,\Gamma)$. Let $X\in\xX^r_T(M,\Gamma,\delta)$. Since hyperbolic critical elements are isolated, there exist only finitely many critical elements of $X$ lying outside of $B_{\delta}(\Gamma)$ and having period $\leq T$. 
  For any point $x\in M\setminus B_{\delta}(\Gamma)$, there are the following cases to consider:
  \begin{enumerate}[(a)]
    \item $x$ is a singularity of $X$;
    \item $\mathrm{Orb}(x)$ is a periodic orbit lying outside of $B_{\delta}(\Gamma)$ and having period $\leq T$;
    \item $\mathrm{Orb}(x)$ is a periodic orbit with period $\leq T$, but it intersects $B_{\delta}(\Gamma)$;
    \item $\mathrm{Orb}(x)$ is not periodic or it is periodic with period $>T$.
  \end{enumerate}

  In case (a), by hyperbolicity of the singularities, there is a neighborhood $U_x\subset M$ and a neighborhood $\cU_x\subset\xX^r(M)$ of $X$ such that every vector field $Y\in\cU_x$ has only one singularity in $U_x$ and it is hyperbolic. Moreover, any periodic orbit of $Y$ that intersects $U_x$ has period $>T$. See \cite[Chapter 3, Lemma 2.1]{PdM}.

  In case (b), the periodic orbit $\mathrm{Orb}(x)$ is hyperbolic. By hyperbolicity, there exists a neighborhood $U_x\subset M$ of $\mathrm{Orb}(x)$ and $\cU_x\subset\xX^r(M)$ such that any $Y\in\cU_x$ has only one periodic orbit $\gamma_Y$ in $U_x$. Moreover, $\gamma_Y$ is hyperbolic and all other periodic orbits of $Y$ that intersects $U_x$ have period $>T$. Furthermore, shrinking $U_x$, $Y$ has no singularities in $U_x$.

  In case (c), by the Tubular Flow Theorem, there is a neighborhood $U_x\subset M$ of $x$ such that any periodic orbit of $X$ that intersects $\overline{U_x}$ will also intersect $B_{\delta}(\Gamma)$. As the flow depends continuously on the vector field, there exists a neighborhood $\cU_x\subset \xX^r(M)$ of $X$ such that the same property holds for all $Y\in\cU_x$. Moreover, we can assume that each $Y\in\cU_x$ has no singularities in $U_x$.

  In case (d), by the Tubular Flow Theorem, there is a neighborhood $U_x\subset M$ of $x$ such that any periodic orbit of $X$ that intersects $\overline{U_x}$ has period $>T$ and $X$ has no singularities in $\overline{U_x}$. It follows that there exists a neighborhood $\cU_x\subset \xX^r(M)$ of $X$ such that every vector field $Y\in\cU_x$ has no singularities in $\overline{U_x}$ and its periodic orbits intersecting $\overline{U_x}$ have period $>T$. See \cite[Chapter 3, Lemma 2.3]{PdM}.

  Since $\{U_x: x\in M\setminus B_{\delta}(\Gamma)\}$ is an open cover of the compact set $M\setminus B_{\delta}(\Gamma)$, we can choose a finite set of points $x_1,\ldots,x_k$ such that $\{U_{x_1},\ldots,U_{x_k}\}$ forms a finite subcover. Let $\cU=\bigcap_{i=1}^k\cU_{x_1}$.
  Then $\widetilde{\cU}:=\cU\cap\xX^r(M,\Gamma)$ is a neighborhood of $X$ in $\xX^r(M,\Gamma)$. For any $Y\in\widetilde{\cU}$, one can see that the singularities of $Y$ lying outside of $B_{\delta}(\Gamma)$ are still hyperbolic. Moreover, any periodic orbit $\gamma_Y$ of $Y$ with period $\leq T$ and lying outside of $B_{\delta}(\Gamma)$ is a continuation of some hyperbolic periodic orbit of $X$ and $\gamma_Y$ remains hyperbolic. This shows that $\xX^r_T(M,\Gamma,\delta)$ contains $\widetilde{\cU}$, and hence is open in $\xX^r(M,\Gamma)$.

  It remains to show that $\xX^r_T(M,\Gamma,\delta)$ is dense in $\xX^r(M,\Gamma)$. Let us choose a $C^{\infty}$ bump function $b_{\delta}:M\to \mathbb{R}$ such that $0\leq b_{\delta}(x)\leq 1$ for all $x\in M$ and
  \begin{equation*}
    b_{\delta}(x)=\begin{cases}
      1,\quad \text{if}\ x\in \Gamma\\
      0,\quad \text{if}\ x\notin B_{\delta}(\Gamma).\\
    \end{cases}
  \end{equation*}
  Let $\xX^r_T(M)\subset\xX^r(M)$ be the set of vector fields whose critical elements with period $\leq T$ are all hyperbolic. It is known that $\xX^r_T(M)$ is open and dense in $\xX^r(M)$, see \cite[Chapter 2, Theorem 2.6]{PdM}. For any $X\in\xX^r_T(M,\Gamma)$, we can take $Y\in\xX^r_T(M)$ arbitrarily $C^r$ close to $X$ and define the vector field
  \[\hat{Y}(x)=X(x)+(1-b_{\delta}(x))(Y(x)-X(x)).\]
  It holds that $\hat{Y}$ is $C^r$ and
  \begin{gather*}
    \hat{Y}|_{\Gamma}=X|_{\Gamma},\\
    \hat{Y}|_{M\setminus B_{\delta}(\Gamma)}=Y|_{M\setminus B_{\delta}(\Gamma)}.
  \end{gather*}
  Hence $\hat{Y}\in\xX^r_T(M,\Gamma,\delta)$. Since the $C^r$-norm of $(1-b_{\delta})$ is a constant independent of $Y$ and $Y$ can be arbitrarily $C^r$ close to $X$, we see that $\hat{Y}$ can also be arbitrarily $C^r$ close to $X$. This proves the density of $\xX^r_T(M,\Gamma,\delta)$. The proof is complete now.
\end{proof}

Recall that for any hyperbolic critical element $c$ of $X$, its stable set $W^s(c,X)$ and unstable set $W^u(c,X)$ are defined as
\[W^s(c,X)=\{x\in M: \omega(x)=c\},\quad W^u(c,X)=\{x\in M: \alpha(x)=c\}.\]
The Stable Manifold Theorem states that $W^s(c,X)$ and $W^u(c,X)$ are in fact $C^r$ immersed submanifolds of $M$.

For each $X\in\xX^r_T(M,\Gamma,\delta)$, let $c_1,\ldots,c_s$ be the critical elements of $X$ lying outside of $B_{\delta}(\Gamma)$ and having period $\leq T$. For each $i$ let us take compact neighborhoods $W^s_0(c_i,X)$ and $W^u_0(c_i,X)$ of $c_i$ in $W^s(c_i,X)$ and $W^u(c_i,X)$, respectively. 
Moreover, we may require that the submanifolds $W^s_0(c_i,X)$ and $W^u_0(c_i,X)$ are contained in $M\setminus \overline{B_{\delta/2}(\Gamma)}$ for all $i$.

By hyperbolicity of the critical elements, there exists a neighborhood $\cU\subset \xX^r(M)$ of $X$ such that each $c_i$ has a continuation $c_{i,Y}$ for any $Y\in\cU$, and moreover, $Y$ has no other critical elements lying outside of $B_{\delta}(\Gamma)$ and having period $\leq T$. In particular, any critical element of $Y\in\cU$, which lies outside of $B_{\delta}(\Gamma)$ and has period $\leq T$, must be $c_{i,Y}$ for some $i$.
By the Stable Manifold Theorem, there are continuous families of local stable manifolds $\{W^s_0(c_{i,Y},Y)\}_{Y\in\cU}$ and local unstable manifolds $\{W^u_0(c_{i,Y},Y)\}_{Y\in\cU}$ for each $i$. Shrinking $\cU$ if necessary, we can assume that $W^s_0(c_{i,Y},Y)$ and $W^u_0(c_{i,Y},Y)$ are contained in $M\setminus \overline{B_{\delta/2}(\Gamma)}$ for all $1\leq i\leq s$ and $Y\in\cU$.

For each $n\in\mathbb{N}_+$, we define $W^s_n(c_{i,Y},Y)=\phi^Y_{-n}(W^s_0(c_{i,Y},Y))$ and $W^u_n(c_{i,Y},Y)=\phi^Y_n(W^u_0(c_{i,Y},Y))$. Then $W^s_n(c_{i,Y},Y)$ varies continuously with respect to $Y$ and so does $W^u_n(c_{i,Y},Y)$. Moreover, $W^s_n(c_{i,Y},Y)$ and $W^u_n(c_{i,Y},Y)$ are compact submanifolds with boundary, and
\[W^s(c_{i,Y},Y)=\bigcup_{n\geq0}W^s_n(c_{i,Y},Y),\quad W^u(c_{i,Y},Y)=\bigcup_{n\geq0}W^u_n(c_{i,Y},Y).\]
Let $\xX^r_{T,n}(M,\Gamma,\delta)$ be the set of vector fields $Y\in\cU\cap\xX^r_T(M,\Gamma,\delta)$ such that $W^s_n(c_{i,Y},Y)$ is transverse to $W^u_n(c_{j,Y},Y)$ for all $i,j\in\{1,\ldots,s\}$.

\begin{lemma}\label{lem.trans-inv-mflds}
  Let $X\in\xX^r_T(M,\Gamma,\delta)$ and $\cU$ be a neighborhood of $X$ as above. Then for each $n\in\mathbb{N}$, $\xX^r_{T,n}(M,\Gamma,\delta)$ is open and dense in $\cU\cap\xX^r_T(M,\Gamma,\delta)$.
\end{lemma}
\begin{proof}
  Let $c_1,\ldots,c_s$ be the critical elements of $X$ lying outside of $B_{\delta}(\Gamma)$ and having period $\leq T$. For any $1\leq i,j\leq s$, we denote by $\cU_{n}^{i,j}$ the set of vector fields $Y\in\cU\cap\xX^r_T(M,\Gamma,\delta)$ such that $W^s_n(c_{i,Y},Y)$ is transverse to $W^u_n(c_{j,Y},Y)$. One has $\xX^r_{T,n}(M,\Gamma,\delta)=\bigcap_{i,j=1}^s\cU_n^{i,j}$. Thus, it suffices to show that $\cU_n^{i,j}$ is open and dense in $\cU\cap\xX^r_T(M,\Gamma,\delta)$.
  That $\cU_n^{i,j}$ is open in $\cN$ follows easily from the transverse property and continuity of the local stable manifold $W^s_n(c_{i,Y},Y)$ and the local unstable manifold $W^u_n(c_{i,Y},Y)$. In the following we show density of $\cU^{i,j}_n$.

  Let $Y\in\cU\cap\xX^r_T(M,\Gamma,\delta)$.
  Recall that $\Gamma$ is a compact invariant subset.
  It follows that given $n\in\mathbb{N}$, there exists $\delta'\in (0,\delta/2)$ such that $W^s_n(c_{i,Y},Y)$ and $W^u_n(c_{j,Y},Y)$ are contained in $M\setminus \overline{B_{\delta'}(\Gamma)}$. We fix a $C^{\infty}$ bump function $b_{\delta'}:M\to\mathbb{R}$ such that $0\leq b_{\delta'}(x)\leq 1$ and
  \begin{equation*}
    b_{\delta'}(x)=\begin{cases}
      1,\quad \text{if}\ x\in \Gamma,\\
      0,\quad \text{if}\ x\notin B_{\delta'}(\Gamma).
    \end{cases}
  \end{equation*}
  Let $C=\|1-b_{\delta'}\|_{C^r}$.

  For any $\vep>0$, by Kupka-Smale theorem we can take a Kupka-Smale $Y'\in\cU$ such that $\|Y-Y'\|_{C^r}<C^{-1}\vep$.
  Moreover, the Stable Manifold Theorem implies that we can suppose $W^s_n(c_{i,Y'},Y')$ and $W^u_n(c_{j,Y'},Y')$ are contained in $M\setminus \overline{B_{\delta'}(\Gamma)}$. Note that $W^s_n(c_{i,Y'},Y')$ is transverse to $W^u_n(c_{j,Y'},Y')$ since $Y'$ is Kupka-Smale.
  Define
  \[\hat{Y}(x)=Y(x)+(1-b_{\delta'}(x))(Y'(x)-Y(x)).\]
  We have
  \begin{gather*}
    \hat{Y}|_{\Gamma}=Y|_{\Gamma},\\
    \hat{Y}|_{M\setminus B_{\delta'}(\Gamma)}=Y'|_{M\setminus B_{\delta'}(\Gamma)}.
  \end{gather*}
  It follows that $\hat{Y}\in\xX^r_{T}(M,\Gamma,\delta)$ and $W^s_n(c_{i,\hat{Y}},\hat{Y})$ is transverse to $W^u_n(c_{i,\hat{Y}},\hat{Y})$ as they coincide with $W^s_n(c_{i,Y'},Y')$ and $W^u_n(c_{j,Y'},Y')$, respectively.
  By the way $\hat{Y}$ is defined, we have \[\|\hat{Y}-Y\|_{C^r}\leq\|1-b_{\delta'}\|_{C^r}\cdot\|Y'-Y\|_{C^r}<\vep.\]
  As $\vep$ is arbitrary, we can assume without loss of generality that $\hat{Y}\in\cU$. Then, $\hat{Y}\in \cU_n^{i,j}$. This proves that $\cU_n^{i,j}$ is dense in $\cU\cap\xX^r_T(M,\Gamma,\delta)$.
\end{proof}

Now we are ready to show that the set $KS(M,\Gamma)$ of $\Gamma$-avoiding Kupka-Smale vector fields is residual in $\xX^r(M,\Gamma)$.
\begin{proof}
  [Proof of Theorem \ref{thm.Gamma-KS}]
  Given $T>0$ and $\delta>0$. Let $\xX^r_{T,*}(M,\Gamma,\delta)$ be the set of vector fields $Y\in\xX^r_T(M,\Gamma,\delta)$ such that $W^s(c_1,Y)$ is transverse to $W^u(c_2,Y)$ whenever $c_1$ and $c_2$ are critical elements of $Y$ lying outside of $B_{\delta}(\Gamma)$ and having period $\leq T$. It follows from Lemma \ref{lem.trans-inv-mflds} that for any $X\in\xX^r_T(M,\Gamma,\delta)$, there is a $C^r$ neighborhood $\cU_X$ of $X$ such that $\xX^r_{T,*}(M,\Gamma,\delta)\cap\cU_X=\bigcap_{n\geq 0}\xX^r_{T,n}(M,\Gamma,\delta)$ is residual in $\xX^r_T(M,\Gamma,\delta)\cap\cU_X$. This implies that $\xX^r_{T,*}(M,\Gamma,\delta)$ is residual in $\xX^r_T(M,\Gamma,\delta)$, see \cite[Chapter 3, Lemma 3.3]{PdM}.
  By Lemma \ref{lem.h-critical-elements}, $\xX^r_T(M,\Gamma,\delta)$ is open and dense in $\xX^r(M,\Gamma)$. It follows that $\xX^r_{T,*}(M,\Gamma,\delta)$ is residual in $\xX^r(M,\Gamma)$. Consequently, $KS(M,\Gamma)=\bigcap_{m,k=1}^{\infty}\xX^r_{m,*}(M,\Gamma,1/k)$ is residual in $\xX^r(M,\Gamma)$.
\end{proof}

\subsection{Generic properties in the $C^1$ topology}

This subsection is devoted to the proof of Theorem \ref{thm.genericity}, \ref{thm.f-limit}, and \ref{thm.relative-h-class}. Before the proof, we would like to recall some key definitions and results, such as chain transitivity and $C^1$ connecting lemmas.

\subsubsection{Chain recurrence and connecting lemmas}
For any $x\in M$ and $t>0$, we denote by $(x,t)$ the orbit segment $\{\phi_{s}(x): 0\leq s\leq t\}$.
Given $x,y\in M$, for any constant $\vep>0$, a finite sequence of orbit segments $\{(x_i,t_i)\}_{i=0}^{n-1}$  is called an $\vep$-chain (or $\vep$-pseudo-orbit) from $x$ to $y$ if $t_i\geq 1$ and  $\mathrm{dist}(\phi_{t_i}(x_i),x_{i+1})<\vep$ for all $i=0,\ldots,n-1$, where $x_0=x$, $x_n=y$.
\begin{definition}
  Let $U$ be any open subset of $M$. Suppose $x,y\in M\setminus U$. We say that {\em $y$ is chain attainable from $x$ avoiding $U$} if for every $\vep>0$, there exists an $\vep$-chain from $x$ to $y$ that has no intersection with $U$.
\end{definition}
The following is a connecting lemma for chains.
\begin{lemma}[\cite{BC}]
\label{lem.connecting}
  For any $\Gamma$-avoiding Kupka-Smale $X\in \xX^1(M,\Gamma)$, for any $C^1$-neighborhood $\cU$ of $X$, any neighborhood $U$ of $\Gamma$ and $x,y\in M\setminus U$, if $y$ is chain attainable from $x$ avoiding $U$, then there exists $Y\in\cU\cap\xX^1(M,\Gamma)$ and $t>0$ such that $\phi^Y_t(x)=y$.
\end{lemma}
The statement of the lemma is somewhat different from \cite{BC}, but it is essentially contained there. The key observation is that the perturbations for connecting will be taken in arbitrarily small neighborhoods of the chains, hence away from $\Gamma$. See \cite[Theorem 6]{Cr06} for a similar statement of the result.

A compact invariant set $\Lambda$ is called {\em chain-transitive} if for any $x,y\in \Lambda$ and any $\vep>0$, there exists an $\vep$-chain from $x$ to $y$ which is contained in $\Lambda$. In particular, if $\Lambda$ is chain-transitive and $\Lambda\cap\Gamma=\emptyset$, then there is a neighborhood $U$ of $\Gamma$ such that for any $x,y\in\Lambda$, $y$ is chain attainable from $x$ avoiding $U$. The following result is an easy consequence of Lemma \ref{lem.connecting} and its proof can be found in \cite[Section 2.4]{Cr06}.

\begin{corollary}[{\cite[Theorem 3]{Cr06}}]
\label{cor.chain-transitive}
  For any $\Gamma$-avoiding Kupka-Smale $X\in\xX^1(M,\Gamma)$, for any compact chain-transitive set $\Lambda$ of $X$, if $\Lambda\cap\Gamma=\emptyset$, then there exists a sequence of vector fields $X_n\in\xX^1(M,\Gamma)$, each with a periodic orbit $\gamma_n$, such that $\gamma_n$ converges to $\Lambda$ in the Hausdorff topology.
\end{corollary}

A chain-transitive set consists of {\em chain-recurrent} points: a point $x$ is called chain-recurrent if there exist $\vep$-chains from $x$ to itself for any $\vep>0$.
We say that a chain-transitive set $\Lambda$ is {\em maximal} if it is the only chain-transitive set that contains $\Lambda$. When a chain-transitive set is maximal, it is called a {\em chain recurrence class}. For any chain-recurrent point $x$, we denote by $C(x)$ the chain recurrence class that contains $x$. A chain recurrence class is said to be {\em non-trivial} if it is not reduced to a critical element.

We will also need the following version of connecting lemma.
\begin{lemma}[{\cite{Ar,WX}}]
\label{lem.connecting-wx}
  For any vector field $X\in\xX^1(M)$, for any $z\notin \Sing(X)\cup\Per(X)$, for any $\vep>0$, there exist $L>0$ and two neighborhoods $\widetilde{W}_z\subset W_z$ of $z$ such that the following property holds.

  For any $p,q\in M$, if the positive orbit of $p$ and the negative orbit of $q$ enter $\widetilde{W}_z$, but the orbit segments $\{\phi^X_t(p):0\leq t\leq L\}$ and $\{\phi^X_t(q): -L\leq t\leq 0\}$ do not intersect $W_z$, then there is a vector field $Y$ $\vep$-close to $X$ in the $C^1$ topology such that
  \begin{itemize}
    \item $q$ is in the positive orbit of $p$ under the flow $\phi^Y_t$ generated by $Y$;
    \item $Y(x)=X(x)$ for any $x\in M\setminus W_{L,z}$, where $W_{L,z}=\bigcup_{0\leq t\leq L}\phi^X_t(W_z)$.
  \end{itemize}
  Moreover, the neighborhoods $\widetilde{W}_z\subset W_z$ of $z$ can be chosen arbitrarily small.
\end{lemma}

\subsubsection{Proof of Theorem \ref{thm.genericity} and \ref{thm.f-limit}}
Let us denote by $\xC$ the set of all nonempty compact subsets of $M$, endowed with the Hausdorff topology. Take a countable basis $V_1,V_2,\ldots$ of $\xC$. In the following, we consider the space $\xX^1(M,\Gamma)$.

  \begin{proof}
    [Proof of Theorem \ref{thm.genericity}]
    For each $m,n\in\mathbb{N}_+$, let $\cH_{m,n}$ be the set of $Y\in \xX^1(M,\Gamma)$ such that $Y$ has a hyperbolic periodic orbit which is an element of $V_n$ and contained in $M\setminus\overline{B_{1/m}(\Gamma)}$. By robustness of hyperbolic periodic orbits, $\cH_{m,n}$ is an open set in $\xX^1(M,\Gamma)$. Let $\cN_{m,n}$ be the interior of the set $\xX^1(M,\Gamma)\setminus \cH_{m,n}$. Then for each $Y\in\cN_{m,n}$, there is a $C^1$ neighborhood $\cU$ of $Y$ in $\xX^1(M,\Gamma)$ such that for each $Z\in\cU$, every hyperbolic periodic orbit of $Z$ contained in $M\setminus\overline{B_{1/m}(\Gamma)}$ is not an element of $V_n$. Define
    \[\cG=\bigcap_{m,n=1}^{\infty}(\cH_{m,n}\cup\cN_{m,n}).\]
    By construction, $\cG$ is a residual subset of $\xX^1(M,\Gamma)$.
    Let $\cG_1=\cG\cap KS(M,\Gamma)$, where $KS(M,\Gamma)$ is the set of $\Gamma$-avoid Kupka-Smale vector fields. Now since $KS(M,\Gamma)$ is residual in $\xX^1(M,\Gamma)$ (Theorem \ref{thm.Gamma-KS}), so is $\cG_1$.

    For any $X\in \cG_1$, suppose $\Lambda$ is a compact chain-transitive set and $\Lambda\cap\Gamma=\emptyset$. Let $V_{n_1}\supset V_{n_2}\supset\cdots$ be a sequence of neighborhoods of $\Lambda$ in $\xC$ such that $\bigcap_k V_{n_k}=\{\Lambda\}$. Since $\Lambda\cap\Gamma=\emptyset$, we can assume that elements in $V_{n_k}$ are uniformly away from $\Gamma$ for all $k$. In particular, there is $m_0$ such that each element of $V_{n_k}$ is contained in $M\setminus\overline{B_{1/m_0}(\Gamma)}$. Since $X\in\cG_1$, we have $X\in (\cH_{m_0,n_k}\cup\cN_{m_0,n_k})\cap KS(M,\Gamma)$. Corollary \ref{cor.chain-transitive} implies that $X\notin\cN_{m_0,n_k}$. Hence $X\in\cH_{m_0,n_k}\cap KS(M,\Gamma)$, and this holds for all $k$. It follows that $X$ has a hyperbolic periodic orbit $\gamma_k\in V_{n_k}$ for each $k$. Since $\bigcap_k V_{n_k}=\{\Lambda\}$,  one finds that $\Lambda$ is the Hausdorff limit of the periodic orbits $\gamma_k$.
  \end{proof}

The proof of Theorem \ref{thm.f-limit} uses a similar argument.
\begin{proof}
  [Proof of Theorem \ref{thm.f-limit}]
  For each $m,n\in\mathbb{N}_+$, let $\cH'_{m,n}$ be the set of $Y\in\xX^1(M,\Gamma)$ such that $Y$ has an index $i$ hyperbolic periodic orbit which is an element of $V_n$ and contained in $M\setminus\overline{B_{1/m}(\Gamma)}$. By robustness of hyperbolic periodic orbits, $\cH'_{m,n}$ is an open set in $\xX^1(M,\Gamma)$. Let $\cN'_{m,n}$ be the interior of the set $\xX^1(M,\Gamma)\setminus\cH'_{m,n}$. Then for each $Y\in\cN'_{m,n}$, there is a $C^1$ neighborhood $\cU$ of $Y$ in $\xX^1(M,\Gamma)$ such that for each $Z\in\cU$, every index $i$ hyperbolic periodic orbit of $Z$ contained in $M\setminus \overline{B_{1/m}(\Gamma)}$ is not an element of $V_n$. Define
  \[\cG'=\bigcap_{m,n=1}^{\infty}(\cH'_{m,n}\cup\cN'_{m,n}).\]
  Then $\cG'$ is a residual subset of $\xX^1(M,\Gamma)$. It follows that $\cG_2:=\cG'\cap KS(M,\Gamma)$ is also a residual subset of $\xX^1(M,\Gamma)$.

  For any $X\in \cG_2$, suppose $\Lambda$ is an index $i$ fundamental limit within $\xX^1(M,\Gamma)$ and $\Lambda\cap\Gamma=\emptyset$.
  Let $V_{n_1}\supset V_{n_2}\supset\cdots$ be a sequence of neighborhoods of $\Lambda$ in $\xC$ such that $\bigcap_k V_{n_k}=\{\Lambda\}$. Since $\Lambda\cap\Gamma=\emptyset$, we can assume that elements in $V_{n_k}$ are uniformly away from $\Gamma$ for all $k$. In particular, there is $m_0$ such that each element of $V_{n_k}$ is contained in $M\setminus\overline{B_{1/m_0}(\Gamma)}$. Since $X\in\cG_2$, we have $X\in (\cH'_{m_0,n_k}\cup\cN'_{m_0,n_k})\cap KS(M,\Gamma)$. Since $\Lambda$ is an index $i$ fundamental limit within $\xX^1(M,\Gamma)$, by definition there exists $X_n\to X$ in the $C^1$ topology and each $X_n\in\xX^1(M,\Gamma)$ has an index $i$ hyperbolic periodic orbit $\gamma_n$ such that $\gamma_n\to\Lambda$ in the Hausdorff topology. This implies that $X\notin\cN'_{m_0,n_k}$. Hence $X\in\cH'_{m_0,n_k}\cap KS(M,\Gamma)$, and this holds for all $k$. It follows that $X$ has an index $i$ hyperbolic periodic orbit $\gamma_k\in V_{n_k}$ for each $k$. Since $\bigcap_k V_{n_k}=\{\Lambda\}$,  one finds that $\Lambda$ is the Hausdorff limit of the index $i$ periodic orbits $\gamma_k$.
  This completes the proof.
\end{proof}

\subsubsection{Proof of Theorem \ref{thm.relative-h-class}}
Recall that the {\em homoclinic class} of a hyperbolic periodic orbit $\gamma$, denoted by $H(\gamma)$, is defined to be the closure of the set of transverse intersections of $W^s(\gamma)$ and $W^u(\gamma)$.
For any neighborhood $U$ of $\gamma$, one defines $W^s_U(\gamma)$ to be the set of points $x\in W^s(\gamma)$ whose entire forward orbit is contained in $U$. Similarly, one defines $W^u_U(\gamma)$ to be the set of points $x\in W^u(\gamma)$ whose entire backward orbit is contained in $U$. Then the relative homoclinic class $H(\gamma, U)$ of $\gamma$ in $U$ is defined as
\[H(\gamma,U)=\overline{W^s_U(\gamma)\pitchfork W^u_U(\gamma)}.\]
In other words, $H(\gamma,U)$ is the closure of the transverse homoclinic intersections of $\gamma$ whose entire orbits are contained in $U$.
Clearly, the relative homoclinic class $H(\gamma,U)$ is a closed subset of $H(\gamma)$. 

Now, let us prove Theorem \ref{thm.relative-h-class}.
\begin{proof}
  [Proof of Theorem \ref{thm.relative-h-class}]
  Since the closure of the unstable manifold of a hyperbolic periodic orbit varies lower semi-continuously, there exists a residual subset $\cG''$ of $\xX^1(M,\Gamma)$ such that $\overline{W^u(\gamma,X)}$ varies continuously with respect to the vector field for each hyperbolic periodic orbit $\gamma$ of $X\in\cG''$ such that $\gamma\cap\Gamma=\emptyset$. By the same reasoning, we may assume that $\overline{W^s(\gamma,X)}$ also varies continuously with respect to the vector field.

  Let $\cG_1$ be the residual subset of $\xX^1(M,\Gamma)$ given by Theorem \ref{thm.genericity}. For each $X\in\cG_1\cap\cG''$, suppose $\Lambda$ is a compact chain transitive set such that $\Lambda\cap\Gamma=\emptyset$ and it contains a periodic orbit $\gamma$, for any neighborhood $U$ of $\Lambda$, we claim that $\Lambda\subset  \overline{W^u_U(\gamma,X)}\cap\overline{W^s_U(\gamma,X)}$. Firstly, if $\Lambda=\gamma$, then the claim holds trivially. We now assume $\Lambda\setminus\gamma\neq\emptyset$. Since $\Lambda$ is chain transitive and is not reduced to $\gamma$, there exists a point $x\in \Lambda\cap W^u_U(\gamma,X)\setminus\gamma$. By Theorem \ref{thm.genericity}, $\Lambda$ is the Hausdorff limit of a sequence of periodic orbits. It follows that for any $z\in \Lambda$, there is an orbit segment contained in $U$, from a point arbitrarily close to $x$ to a point  $y\in U$, which is arbitrarily close to $z$. By Lemma \ref{lem.connecting-wx}, there is vector field $Y$ arbitrarily $C^1$ close to $X$ such that $x\in W^u(\gamma,Y)$ and the forward orbit of $x$ goes to $y$ along an orbit segment in $U$. In other words, $y\in  W^u_U(\gamma,Y)$. Observe that $\Lambda$ is away from $\Gamma$, the connecting is also away from $\Gamma$. (We may assume that $\overline{U}\cap \Gamma=\emptyset$ and the perturbations are done in $U$.) Hence $Y\in\xX^r(M,\Gamma)$. As $X\in\cG''$ and $y$ is arbitrarily close to $z$, the continuity of $\overline{W^u(\gamma,X)}$ implies that $z\in \overline{W^u(\gamma,X)}$. In fact, the proof above shows that $z\in \overline{W^u_U(\gamma,X)}$. Similarly, $z\in\overline{W^s_U(\gamma,X)}$. This proves the claim.

  Then, for any $z\in\Lambda$ and any neighborhood $U_z$ of $z$, using Lemma \ref{lem.connecting-wx} again one can obtain a vector field $X'$ arbitrarily $C^1$  close to $X$ such that there is a homoclinic point $w\in W^u_U(\gamma,X')\cap W^s_U(\gamma,X')$ of $\gamma$ contained in $U_z$. By an extra perturbation if necessary, we can suppose that $w$ is an transverse intersection of $W^s_U(\gamma,X')$ and $W^u_U(\gamma,X')$. Note again that the perturbations for the connecting is away from $\Gamma$, we have $X'\in\xX^1(M,\Gamma)$.

  Now, since transverse homoclinic points persist under small perturbation, by a standard argument we can find a residual subset $\cG_3\subset \cG_1\cap\cG''$ of $\xX^1(M,\Gamma)$ such that for each $X\in\cG_3$, if $\Lambda$ is a compact chain transitive set of $X$ such that $\Lambda\cap\Gamma=\emptyset$ and it contains a periodic orbit $\gamma$, then for any neighborhood $U$ of $\Lambda$, one has $\Lambda\subset \overline{W^s_U(\gamma,X)\pitchfork W^u_U(\gamma,X)}$. In other words, $\Lambda$ is contained in the relative homoclinic class $H(\gamma,U)$.
\end{proof}

\section{Codimension one partially hyperbolic dynamics}
\label{sect.codim-one}

We now study codimension one partially hyperbolic dynamics for generic vector fields in $\xX^1(M,\Gamma)$, with the aim to prove Theorem \ref{thm.codim-one}.

\subsection{Preliminaries on partially hyperbolic dynamics}
\label{sect.preliminaries-ph}

\subsubsection{Linear Poincar\'e flow}
To study the dynamics of a flow, one usually considers the normal bundle of the flow and the induced dynamics therein. 
At a regular point $x$, there exists a normal space $\cN_x=\{v\in T_xM: v\perp X(x)\}$. In general, for any invariant set $\Lambda$ of the flow such that  $\Lambda\cap\Sing(X)=\emptyset$ (such a set will be called {\em nonsingular}), we define the {\em normal bundle} over $\Lambda$ as the following:
\[\cN_{\Lambda}=\bigcup_{x\in\Lambda}\cN_x.\]
On the normal bundle it is defined the {\em linear Poincar\'e flow} $\psi_t$ as follows:
\[\psi_t(v)=\Phi_t(v)-\frac{\langle\Phi_t(v),X(\phi_t(x))\rangle}{|X(\phi_t(x))|^2} X(\phi_t(x)),\quad \forall v\in\cN_x.\]
In other words, for $v\in\cN_x$ and $t\in\mathbb{R}$, one obtains $\psi_t(v)$ as the orthogonal projection of $\Phi_t(v)$ on the normal space at $\phi_t(x)$.
Inspired by the pioneering work of Liao \cite{Lia80,Lia89}, Li-Gan-Wen \cite{LGW05} obtained the {\em extended linear Poincar\'e flow}, which extends the linear Poincar\'e flow ``to singularities''.

\subsubsection{Normal manifolds and local holonomies}
For any $x\in M\setminus\Sing(X)$ and $r>0$, define $\cN_x(r)=\{v\in\cN_x:|v|\leq r\}$ and
\[N_x(r)=\exp_x(\cN_x(r)),\]
where $\exp_x:T_xM\to M$ is the exponential map. One can take $\beta_0>0$ small such that for any $x\in M\setminus\Sing(X)$, the exponential map $\exp_x$ is a diffeomorphism from $\cN_x(\beta_0)$ onto its image.
For $r\in(0,\beta_0]$, we call $N_x(r)$ a {\em normal manifold} (of size $r$) at $x$.

Note that for each $x\in M\setminus\Sing(X)$, there is $r_x>0$ such that the flow $\phi_t$ induces a local holonomy map $P_{x,t}$ from $N_x(r_x)$ to $N_{\phi_t(x)}(\beta_0)$, for any $t\in[-1,1]$.
It is shown in \cite[Lemma 2.4]{GY} that for any $T>0$, there is $\beta_T>0$ such that for any $x\in M\setminus\Sing(x)$, $P_{x,t}$ is well-defined on $N_x(\beta_T|X(x)|)$ for any $t\in[-T,T]$.
Given $x\in M\setminus\Sing(X)$, $t\in(0,T]$, $y\in N_x(\beta_T|X(x)|)$, we denote $\tau_{x}(t,y)$ the smallest positive number such that $\phi_{\tau_x(t,y)}(y)\in N_{\phi_t(x)}(\beta_0)$. Note that we have $\phi_{\tau_x(t,y)}(y)=P_{x,t}(y)$.

\begin{lemma}[{\cite[Lemma 2.7]{PYY2}}]
  \label{lem.tau-lip}
  Given $T>0$. For every $x\in M\setminus\Sing(X)$ and $t\in(0,T]$, $\tau_{x}(t,y)$ is differentiable with respect to $y\in N_x(\beta_T|X(x)|)$; furthermore, there exists $K>0$ such that for any $x\in M\setminus\Sing(X)$ and $t\in(0,T]$, it holds
  \[|X(x)|\cdot \sup\left\{\|D_y\tau_{x}(t,y)\|: y\in N_x(\beta_T|X(x)|)\right\}\leq K.\]
\end{lemma}

\subsubsection{Dominated splitting and invariant manifolds}
\label{sect.invariant-manifolds}
Let $\Lambda$ be a compact invariant set of the flow.
A $\psi_t$-invariant splitting $\cN_{\Lambda\setminus\Sing(X)}=\Delta^{cs}\oplus \Delta^{cu}$ is said to be {\em dominated}, if there exists $T>0$ (which will be called a {\em domination constant}) such that for any unit vectors $v^{cs}\in \Delta^{cs}$, $v^{cu}\in \Delta^{cu}$, it holds
\[|\psi_t(v^{cs})|/|\psi_t(v^{cu})|<1/2,\quad \forall t\geq T.\]
It is well-known that domination implies existence of local invariant manifolds \cite{HPS}. For flows with singularities, one needs to be more careful with the sizes of the invariant manifolds.
\begin{lemma}[{\cite[Lemma 2.15]{GY}}]
\label{lem.plaque-families}
  Let $\Lambda$ be a compact invariant set of $X\in\xX^1(M)$. Suppose there exists a dominated splitting $\cN_{\Lambda\setminus\Sing(X)}=\Delta^{cs}\oplus \Delta^{cu}$ for the linear Poincar\'e flow $\psi_t$. Then for any $T>0$, there exist $\xi>0$ and two families of continuous $C^1$ embedded submanifolds $W^{cs}(x)$, $W^{cu}(x)$, with $x\in\Lambda\setminus\Sing(X)$, such that
  \begin{enumerate}
    \item $W^{cs}(x), W^{cu}(x)\subset N_x(\xi|X(x)|)$;
    \item $T_xW^{cs}(x)=\Delta^{cs}(x)$ and $T_xW^{cu}(x)=\Delta^{cu}(x)$;
    \item for any $\vep>0$, there is $\delta>0$ such that for any $x\in\Lambda\setminus\Sing(X)$, one has
        \[P_{x,T}(W^{cs}_{\delta|X(x)|}(x))\subset W^{cs}_{\vep|X(x)|}(\phi_T(x)),\quad P_{x,T}(W^{cu}_{\delta|X(x)|}(x))\subset W^{cu}_{\vep|X(x)|}(\phi_T(x)),\]
        where $W^{\zeta}_{a|X(x)|}(x)=W^{\zeta}(x)\cap N_x(a|X(x)|)$ for $\zeta\in\{cs,cu\}$ and $a\in (0,\xi]$.
  \end{enumerate}
\end{lemma}
The families $W^{cs}(x)$, $W^{cu}(x)$ are called {\em $cs$-plaques} and {\em $cu$-plaques}, respectively.
Note that when $\Lambda\cap\Sing(X)=\emptyset$, the plaque families both have a uniform size.

The dominated splitting is said to be {\em codimension-one} if $\dim \Delta^{cu}=1$. In this case, the $cu$-plaques are $C^1$ embedded curves.

\subsubsection{Hyperbolicity and partial hyperbolicity}

Let $\Lambda$ be a compact invariant set of $X$. Suppose $\Lambda$ is nonsingular.
A $\psi_t$-invariant subbundle $\Delta\subset \cN_{\Lambda}$ is called {\em contracting for $X$}, if there exist constants $C>0$, $\lambda>0$ such that for any vector $v^s\in \Delta$, it holds
\[|\psi_t(v^s)|\leq C\e^{-\lambda t}|v^s|,\quad\forall t\geq 0.\]
If $\Delta$ is contracting for $-X$, then it is called {\em expanding for $X$}.
If $\Lambda$ is nonsingular and admits a dominated splitting $\cN_{\Lambda}=\Delta^s\oplus \Delta^u$ with respect to $\psi_t$ such that $\Delta^s$ is contracting for $X$ and $\Delta^u$ is expanding for $X$, then $\Lambda$ is called {\em hyperbolic}.

Generalizing the concept of hyperbolicity, 
the set $\Lambda$ is said to admit a {\em partially hyperbolic} splitting with respect to the linear Poincar\'e flow, if there exists a $\psi_t$-invariant splitting $\cN_{\Lambda\setminus\Sing(X)}=\Delta^{s}\oplus \Delta^{cu}$ which is dominated and $\Delta^s$ is contracting for $X$.

One also defines partially hyperbolicity with respect to the tangent flow.  We say that $\Lambda$ is {\em partially hyperbolic with respect to the tangent flow} if there exists a $\Phi_t$-invariant splitting of the tangent bundle $T_{\Lambda}M=E^s\oplus E^{cu}$ such that
\begin{itemize}
  \item $E^s$ is {\em dominated} by $E^{cu}$: there exists $T>0$ such that for any unit vectors $v^s\in E^s$, $v^{cu}\in E^{cu}$, it holds
      \[|\Phi_t(v^s)|/|\Phi_t(v^{cu})|<1/2,\quad \forall t\geq T;\]
  \item $E^s$ is {\em contracting} for $X$: there exist constants $C>0$, $\lambda>0$ such that for any vector $v^s\in E^s$, it holds
      \[|\Phi_t(v^s)|\leq C\e^{-\lambda t}|v^s|,\quad\forall t\geq 0.\]
\end{itemize}
Note that the contracting property of $E^s$ implies that $X(x)\in E^{cu}(x)$ for all $x\in\Lambda$. In fact, suppose $X(x)\notin E^{cu}(x)$ for some regular $x$, then one can show that $\|X(\phi_{t}(x))\|$ grows to infinity as $t\to-\infty$, which is absurd.

It turns out that if $\Lambda$ is nonsingular, the two versions of partial hyperbolicity are equivalent, as shown by the following lemma.
\begin{lemma}\label{lem.partial-hyperbolicity}
  Suppose $\Lambda$ is a compact invariant set of $X$, $\Lambda\cap\Sing(X)=\emptyset$. Then it admits a partially hyperbolic splitting $T_{\Lambda}M=E^s\oplus E^{cu}$ with respect to the tangent flow if and only if its normal bundle admits a partially hyperbolic splitting $\cN_{\Lambda}=N^s\oplus N^{cu}$ with respect to $\psi_t$. Moreover, one has
  \[N^s\oplus \langle X\rangle =E^s\oplus \langle X\rangle, \quad\text{and}\quad N^{cu}\oplus \langle X\rangle =E^{cu}.\]
\end{lemma}
One can refer to Theorem 2.27 in \cite{APa} for a similar result, and its proof can be readily adapted to give a proof of Lemma \ref{lem.partial-hyperbolicity}.

\subsection{Minimally nonhyperbolic set}
A compact invariant set $\Lambda$ is called {\em minimally nonhyperbolic} if $\Lambda$ itself is not hyperbolic, but every proper compact invariant subset of $\Lambda$ is hyperbolic.
\begin{lemma}\label{lem.sinks-accumulation}
  Let $X$ be a generic vector field in $\xX^1(M,\Gamma)$. Suppose $\Lambda$ is a chain transitive compact invariant set of $X$ such that $\Lambda\cap(\Gamma\cup\Sing(X))=\emptyset$.
  If $\Lambda$ admits a codimension-one partially hyperbolic splitting $\cN_{\Lambda}=\Delta^s\oplus \Delta^c$ and is minimally nonhyperbolic, then
  it is the Hausdorff limit of periodic sinks.
\end{lemma}
To prove the lemma, let us first recall the following definitions.
A regular point $x\in M$ is called {\em strongly closable} if for any $C^1$ neighborhood $\cU$ of $X$, for any $\vep>0$, there exist $Y\in\cU$, $L>0$ and a periodic point $y$ of $Y$ such that
\begin{itemize}
  \item $X(z)=Y(z)$ for all $z\in M\setminus \bigcup_{t\in[-L,0]}\phi_t(B(x,\vep))$;
  \item $d(\phi^X_t(x),\phi^Y_t(y))<\vep$ for all $t\in[0,\tau(y)]$, where $\tau(y)$ is the period of $y$.
\end{itemize}
Let $\Sigma(X)$ be the set of strongly closable points of $X$.
The Ergodic Closing Lemma \cite{Ma, We} states that for every $\phi^X_t$-invariant probability measure $\mu$, it holds $\mu(\Sigma(X)\cup\Sing(X))=1$.
\begin{definition}
Let $\mu$ be a $\phi^X_t$-invariant ergodic probability measure, which is not the Dirac measure supported on a singularity. A point $x\in M$ is called a {\em typical point} of $\mu$, if it is strongly closable, $\omega(x)=\mathrm{supp}(\mu)$, and
\[\lim_{T\to+\infty}\delta_{x,T}=\mu,\]
where $\delta_{x,T}$ is the uniform distribution measure supported on the orbit segment $\phi_{[0,T]}(x)$, i.e. for any continuous function $f:M\to\mathbb{R}$ it holds
\[\int f\mathrm{d}\delta_{x,T}=\frac{1}{T}\int_0^T f(\phi^X_t(x))\mathrm{d}t.\]
\end{definition}

By the Ergodic Closing Lemma and the Birkhoff Ergodic Theorem, the set of typical points of $\mu$ has $\mu$-full measure.

We can now prove Lemma \ref{lem.sinks-accumulation}.
\begin{proof}[Proof of Lemma \ref{lem.sinks-accumulation}]
  Since $X$ is a generic vector field in $\xX^1(M,\Gamma)$, it is $\Gamma$-avoiding Kupka-Smale. It follows that $\Lambda$ is not a periodic orbit, since it is not hyperbolic. Moreover, there exists $x_0\in \Lambda$ such that $\log\|\psi^X_t|_{\Delta^c(x_0)}\|\leq 0$ for all $t\geq 0$. Let $\mu$ be any accumulation point of $\{\delta_{x_0,t}\}_{t\geq 0}$, where $\delta_{x_0,t}$ is the uniform distribution measure supported on the orbit segment $\phi^X_{[0,t]}(x_0)$.
  Then $\mu$ is an invariant measure supported on $\Lambda$, such that for all $t\geq 0$ it holds
  \begin{equation}\label{eq.nonhyperbolic1}
    \int\log\|\psi^X_t|_{\Delta^c}\|\mathrm{d}\mu\leq 0.
  \end{equation}
  By ergodic decomposition, we may assume that $\mu$ is ergodic.
  Since $\Lambda$ is chain transitive and minimally nonhyperbolic, the inequality implies that  $\mathrm{supp}(\mu)=\Lambda$: otherwise, $\mathrm{supp}(\mu)$ is a hyperbolic set and $\Delta^c|_{\mathrm{supp}(\mu)}$ is expanding, a contradiction to \eqref{eq.nonhyperbolic1}.

  We extend the partially hyperbolic splitting $\cN_{\Lambda}=\Delta^s\oplus \Delta^c$ continuously to a small neighborhood $U_0$ of $\Lambda$, $\overline{U_0}\cap\Gamma=\emptyset$.
  For any $\vep>0$, let $\delta>0$ and $U_1\subset U_0$ be a neighborhood of $\Lambda$ such that for any $x_1,x_2\in U_1$ satisfying $d(x_1,x_2)<\delta$ and $t\in[0,1]$, it holds
  \begin{equation}\label{eq.nonhyperbolic2}
    \left|\log\|\psi^X_t|_{\Delta^c(x_1)}\|-\log\|\psi^X_t|_{\Delta^c(x_2)}\|\right|<\vep/3.
  \end{equation}
  By continuity of dominated splitting, there exists a neighborhood $U\subset U_1$ of $\Lambda$ and a $C^1$ neighborhood $\cU$ of $X$ such that the splitting $\cN_{\Lambda}=\Delta^s\oplus \Delta^c$ extends to the maximal invariant set $\Lambda_Y$ of $Y$ in $U$, and moreover, the subbundle $\Delta^s_Y$ remains contracting (for $Y$) and
  \begin{equation}\label{eq.nonhyperbolic3}
    \left|\log\|\psi^Y_t|_{\Delta^c_Y(x)}\|-\log\|\psi^X_t|_{\Delta^c(x)}\|\right|<\vep/3,
  \end{equation}
  for all $x\in\Lambda_Y$ and $t\in [0,1]$.

  Let us take a typical point $x$ of $\mu$. There is $T_0>0$ such that for any $T>T_0$, one has
  \begin{equation}\label{eq.nonhyperbolic4}
    \left|\int \log\|\psi^X_t|_{\Delta^c}\|\mathrm{d}\delta_{x,T}-\int \log\|\psi^X_t|_{\Delta^c}\|\mathrm{d}\mu\right|<\vep/3,
  \end{equation}
  and moreover,
  \[\mathrm{d}_H\left(\phi^X_{[0,T]}(x),\Lambda\right)<\vep/2.\]
  Here $\mathrm{d}_H(\cdot,\cdot)$ stands for the Hausdorff distance.
  By the Ergodic Closing Lemma, there exists $Y\in\cU$ and a periodic point  $y\in\Lambda_Y$ such that $d(\phi^X_t(x),\phi^Y_t(y))<\min\{\delta,\vep/2\}$ for $t\in[0,\tau(y)]$. Since $x$ is not periodic, we can assume that $\tau(y)>T_0$. Let us denote by $\gamma$ the orbit of $y$ and $\delta_{\gamma}=\delta_{y,\tau(y)}$. Then
  \[\mathrm{d}_H(\gamma,\Lambda)\leq \mathrm{d}_H\left(\gamma,\phi^X_{[0,\tau(y)]}(x)\right)+ \mathrm{d}_H\left(\phi^X_{[0,\tau(y)]}(x),\Lambda\right)<\vep.\]
  Moreover, for any $t\in[0,1]$, the inequalities \eqref{eq.nonhyperbolic1}$\sim$\eqref{eq.nonhyperbolic4} imply the following:
  \begin{align*}
    \int\log\|\psi^Y_t|_{\Delta^c_Y}\|\mathrm{d}\delta_{\gamma} &= \int\log\|\psi^Y_t|_{\Delta^c_Y}\|\mathrm{d}\delta_{\gamma}- \int\log\|\psi^X_t|_{\Delta^c}\|\mathrm{d}\delta_{\gamma}\\
    &\phantom{\leq}+\int\log\|\psi^X_t|_{\Delta^c}\|\mathrm{d}\delta_{\gamma} -\int\log\|\psi^X_t|_{\Delta^c}\|\mathrm{d}\delta_{x,\tau(y)}\\
    &\phantom{\leq}+\int\log\|\psi^X_t|_{\Delta^c}\|\mathrm{d}\delta_{x,\tau(y)}
    -\int\log\|\psi^X_t|_{\Delta^c}\|\mathrm{d}\mu\\
    &\phantom{\leq}+\int\log\|\psi^X_t|_{\Delta^c}\|\mathrm{d}\mu\\
    &\leq \frac{\vep}{3}+\frac{\vep}{3}+\frac{\vep}{3}+0=\vep.
  \end{align*}
  Observe that by the Ergodic Closing Lemma, the perturbations for closing can be done in $U_0$, hence away from $\Gamma$. This shows that $Y\in\xX^1(M,\Gamma)$.

  By Franks' lemma, the periodic orbit $\gamma$ can be turned into a sink by an $\vep$-perturbation. Since $\vep$ can be arbitrarily small, one concludes that $\Lambda$ is a fundamental limit within $\xX^1(M,\Gamma)$ of periodic sinks. As $X$ is generic in $\xX^1(M,\Gamma)$, Theorem \ref{thm.f-limit} shows that $\Lambda$ is in fact the Hausdorff limit of periodic sinks of $X$ itself.
\end{proof}

\subsection{Central model for partially hyperbolic dynamics}\label{sect.central-model}
The notion of central model was introduced by Crovisier \cite{Cr10} to prove the Weak Palis Conjecture \cite{Pa00} for $C^1$ diffeomorphisms, and has been proved to be a powerful tool for studying partially hyperbolic dynamics with one-dimensional center. 
\begin{definition}
A {\em central model} is a pair $(\hat{K},\hat{f})$, where $\hat{K}$ is a compact metric space, $\hat{f}$ is a continuous map from $\hat{K}\times[0,1]$ to $\hat{K}\times[0,+\infty)$, such that
\begin{itemize}
	\item $\hat{f}(\hat{K}\times\{0\})=\hat{K}\times\{0\}$;
	\item $\hat{f}$ is a local homeomorphism in a small neighborhood of $\hat{K}\times\{0\}$;
	\item $\hat{f}$ is a skew product, {\itshape i.e.} there exist $\hat{f}_1:\hat{K}\to\hat{K}$ and $\hat{f}_2:\hat{K}\times[0,1]\to[0,+\infty)$ such that
	\[\hat{f}(x,t)=(\hat{f}_1(x),\hat{f}_2(x,t)),\quad \forall (x,t)\in\hat{K}\times[0,1].\]
\end{itemize}
As $\hat{f}$ is a skew product, the set $\hat{K}$ (naturally identified with $\hat{K}\times\{0\}$) is called the {\em base} of the central model.
\end{definition}

One says that a central model $(\hat{K},\hat{f})$ has a {\em chain-recurrent central segment} if it contains a non-trivial segment $I=\{x\}\times[0,a]$ that is contained in a chain-transitive set of $\hat{f}$. A chain-recurrent central segment is contained in the maximal invariant set in $\hat{K}\times[0,1]$. A {\em trapping strip} for $\hat{f}$ is an open set $S\subset \hat{K}\times [0,1]$ satisfying $\hat{f}(\overline{S})\subset S$, and moreover, $S\cap(\{x\}\times[0,1])$ is a non-trivial interval containing $(x,0)$ for any $x\in \hat{K}$. 
As $\hat{f}$ is a local homeomorphism, its inverse $\hat{f}^{-1}$ is well-defined in a neighborhood of $\hat{K}\times\{0\}$. A trapping strip for $\hat{f}^{-1}$ is called an {\em expanding strip} for $\hat{f}$.

The following lemma establishes a dichotomy between existence of chain-recurrent central segments and existence of arbitrarily small trapping (or expanding) strips.

\begin{lemma}[{\cite[Proposition 2.5]{Cr10}}]\label{lem.central-model-dichotomy}
	Let $(\hat{K},\hat{f})$ be a central model with a chain-transitive base, then
	\begin{itemize}
		\item either there exists a chain-recurrent central segment;
		\item or there exist some trapping strips $S$ in $\hat{K}\times[0,1]$ either for $\hat{f}$ or for $\hat{f}^{-1}$ in arbitrarily small neighborhoods of $\hat{K}\times\{0\}$.
	\end{itemize}
\end{lemma}
\begin{remark}\label{rmk.dichotomy-central-model}
  If there exists a chain-recurrent central segment, then for each $0<a<1$, there exists a chain-recurrent central segment contained in the maximal invariant set of $\hat{K}\times[0,a]$.
\end{remark}

\subsubsection{Central model for codimension-one partially hyperbolic dynamics}

Let $X\in\xX^1(M)$ and $K$ be a nonsingular compact invariant set of $X$. Assume that $K$ admits a codimension-one partially hyperbolic splitting $\cN_{K}=\Delta^s\oplus \Delta^c$ with respect to $\psi_t$.
Let $T>0$ be a domination constant of the splitting.
Since $K$ is nonsingular and compact, there is $\delta_0>0$ such that the local holonomy map $P_{x,T}$ is well-defined on the normal manifold $N_x(\delta_0)$ for $x\in K$.
One thus defines a map $P_T$ on the family of normal manifolds $\{N_x(\delta_0)\}_{x\in K}$ such that
\[P_T|_{N_x(\delta_0)}=P_{x,T}.\]
\begin{definition}
A central model $(\hat{K},\hat{f})$ is called a {\em central model for $(K,P_T)$} if there exists a continuous map $\pi:\hat{K}\times[0,+\infty)\to M$ such that
\begin{itemize}
	\item $\pi$ semi-conjugates $\hat{f}$ and $P_T$, i.e. $P_T\circ\pi=\pi\circ\hat{f}$ on $\hat{K}\times[0,1]$;
	\item $\pi(\hat{K}\times\{0\})=K$;
	\item the collection of maps $t\mapsto \pi(\hat{x},t)$ is a continuous family of $C^1$-embeddings of $[0,+\infty)$ into $M$, parameterized by $\hat{x}\in\hat{K}$;
	\item for any $\hat{x}\in\hat{K}$, the curve $\pi(\hat{x},[0,+\infty))$ is tangent to $\Delta^c(x)$ at the point $x=\pi(\hat{x},0)$, and $\pi(\hat{x},[0,+\infty))\subset N_x(\delta_0)$.
\end{itemize}%
\end{definition}
\begin{remark}\label{rmk.central-model-0}
The set $\hat{K}\times\{0\}$ will be naturally identified with $\hat{K}$. The flow $\phi_t$ acts naturally on $\hat{K}$ such that $\pi(\phi_t(\hat{x}))=\phi_t(x)$ for all $t\in\mathbb{R}$.
We sometimes write $(\hat{K},\hat{f};\pi)$ as the central model for $(K,P_T)$, to emphasize the role played by the map $\pi$.
\end{remark}
\begin{remark}\label{rmk.central-model-1}
One can require that a central model for $(K,P_T)$ is ``small'' in the sense that for any prefixed constant $\delta\in(0,\delta_0]$ and any cone field $\cC$ of the center bundle $\Delta^c$ (continuously extended to a small neighborhood of $K$), for any $\hat{x}\in\hat{K}$, one has $\pi(\hat{x},[0,1])\subset N_{\pi(\hat{x})}(\delta)$, and it is tangent to $\cC$. 
Any curve tangent to $\cC$ will be called a central curve. In particular,  $\pi(\hat{x},[0,1])$ is a central curve at $x=\pi(\hat{x})$.
\end{remark}
\begin{remark}
  The semi-conjugacy between $\hat{f}$ and $P_T$ gives the following invariance property of central curves: \[P_{x,t}(\pi(\hat{x},[0,1]))\subset\pi(\phi_t(\hat{x}),[0,+\infty)),\quad \text{where $t$ is a multiple of $T$}.\]
  However, this may not hold for an arbitrary $t$. In Subsection \ref{sect.central-curves-p} we will see special cases where the invariance does hold for all $t\in\mathbb{R}$.
\end{remark}

Central models for $(K,P_T)$ can be constructed from the center plaque families.
By Lemma \ref{lem.plaque-families}, there exists a $cu$-plaque family $\{W^c(x)\}_{x\in K}$ with uniform size, tangent to $\Delta^c$.
Moreover, the $cu$-plaques (which will also be called {\em center plaques})
are invariant for $P_{x,T}$ in the sense that for any $\vep>0$, there exists $\delta>0$ such that for any $x\in K$,
\[P_{x,T}(W^c_{\delta}(x))\subset W^c_{\vep}(x),\]
where $W^{c}_a(x)=W^c_{\delta_0}(x)\cap N_x(a)$ with $a>0$ small. We may assume that the center plaques $W^c(x)$ are contained in $N_x(\delta_0)$ for all $x\in K$. Note that the center plaques $W^c(x)$ are $C^1$ embedded curves since $\dim \Delta^c=1$.
Then a central model $(\hat{K},\hat{f};\pi)$ can be defined such that $\pi(\hat{x},[0,1])$ is a half of the curve $W^c_{\delta}(x)$ with one endpoint $\pi(\hat{x})=x$ and $\delta>0$ small.

A detailed construction depends on whether $P_T$ preserves an orientation of the center bundle $\Delta^c$ or not, see \cite{Cr10,XZ}. One needs to consider the following two cases:
\begin{itemize}
	\item The {\em orientable} case: the bundle $\Delta^c$ is orientable;
	\item The {\em non-orientable} case: the bundle $\Delta^c$ is not orientable;
\end{itemize}
Note that in the orientable case, $P_T$ preserves the orientation of $N^c$ as it is homotopic to the identity. The following theorem gives the existence of central models for $(K,P_T)$.

\begin{theorem}[{\cite[Section 3.2]{Cr10}, \cite[Section 5.2]{XZ}}]\label{thm.central-model}
	Suppose there is a nonsingular compact $\phi_t$-invariant set $K$ admitting a partially hyperbolic splitting $\cN_K=\Delta^{s}\oplus \Delta^c$ with respect to the linear Poincar\'e flow such that $\dim \Delta^c=1$. Let $T>0$ be the domination constant for the splitting.
	\begin{itemize}
		\item In the orientable case, one can obtain two central models for $(K, P_T)$, which are denoted by $(\hat{K}_+,\hat{f}_+;\pi_+)$ and $(\hat{K}_-,\hat{f}_-;\pi_-)$, such that
		\begin{itemize}
			\item the maps $\pi_{\iota}:\hat{K}_{\iota}\cong\hat{K}_{\iota}\times \{0\}\to K$ are both homeomorphisms, for $\iota\in\{+,-\}$;
			\item for any point $x\in K$, $\iota\in\{+,-\}$, let $\hat{x}^{\iota}\in\hat{K}_{\iota}$ such that $\pi_{\iota}(\hat{x}^{\iota})=x$, then $x$ is the common endpoint of the two central curves $\pi_{+}(\hat{x}^{+},[0,+\infty))$, $\pi_{-}(\hat{x}^{-},[0,+\infty))$.
		\end{itemize}
		
		\item In the non-orientable case, there exists a central model $(\hat{K},\hat{f};\pi)$ for $(K,P_T)$ satisfying the following:
		\begin{itemize}
			\item $\pi:\hat{K}\cong\hat{K}\times\{0\}\to K$ is two-to-one: any point $x\in K$ has exactly two preimages $\hat{x}^+$ and $\hat{x}^-$ in $\hat{K}$;
			\item each point $x\in K$ is the common endpoint of the central curves $\pi(\hat{x}^{+},[0,+\infty))$ and $\pi(\hat{x}^{-},[0,+\infty))$.
		\end{itemize}
		
		\item If $K$ is chain-transitive, then in both cases and for any of the central models, the base is also chain-transitive.
	\end{itemize}
\end{theorem}

\subsubsection{Central curves at periodic points}
\label{sect.central-curves-p}

Let $(\hat{K},\hat{f};\pi)$ be a central model for $(K,P_T)$. 
Suppose there is a trapping strip $S$ for $\hat{f}$, i.e. $S$ is an open set contained in $\hat{K}\times[0,1]$ satisfying $\hat{f}(\overline{S})\subset S$. We denote
\[S^*=\bigcap_{n\geq 0}\hat{f}^n(\mathrm{Cl}(S)),\]
which is an $\hat{f}$-invariant set. For each $\hat{x}\in\hat{K}$, define 
\[\sigma_{\hat{x}}=\pi(S\cap(\{\hat{x}\}\times[0,1])),\quad \sigma^*_{\hat{x}}=\pi(S^*\cap(\{\hat{x}\}\times[0,1])).\]
Note that if $\pi(\hat{x})=x$ is a periodic point, then $\hat{x}$ is also a periodic point (for the induced flow $\phi_t$ on $\hat{K}$, see Remark \ref{rmk.central-model-0}). With the notations as above, the following result holds.
\begin{lemma}[{\cite[Proposition 5.4]{XZ}}]
\label{lem.periodic-c-curve}
  Suppose all periodic points in $\pi(\hat{K}\times[0,1])$ are hyperbolic. Let $\hat{x}\in \hat{K}$ such that $x=\pi(\hat{x})$ is a periodic point. Then $\sigma^*_{\hat{x}}$ is periodic in the sense that $P_{x,\tau}(\sigma^*_{\hat{x}})=\sigma^*_{\hat{x}}$, where $\tau$ is the period of $\hat{x}$. Moreover, on $\sigma^*_{\hat{x}}$ there is a finite alternate arrangement of hyperbolic saddles and sinks $p_0=x,p_1,\ldots,p_n$ with the following properties.
  \begin{itemize}
    \item Let $(p_k,p_{k+1})$ be the connected component of $\sigma^*_{\hat{x}}\setminus\{p_0,p_1,\ldots,p_n\}$ with endpoints $p_k$ and $p_{k+1}$ ($0\leq k\leq n-1$), then $(p_k,p_{k+1})$ is contained in the stable set of one endpoint (sink) and in the unstable set of the other endpoint (saddle).
    \item $p_n$ is the endpoint of $\sigma^*_{\hat{x}}$ other than $x$ and is a periodic sink whose stable set contains the segment $\sigma_{\hat{x}}\setminus\sigma^*_{\hat{x}}$.
  \end{itemize}
  Symmetric results hold when $\hat{f}$ has an expanding strip.
\end{lemma}

\begin{remark}
  In the proof of the lemma, it is shown that the central curves inside $\pi(S^*)$ along any periodic orbit $\mathrm{Orb}(x)$ are invariant for $P_{x,t}$, i.e.
  \begin{equation}\label{eq.c-invariance}
    P_{x,t}(\sigma^*_{\hat{x}})=\sigma^*_{\phi_t(\hat{x})},\quad \forall t\in\mathbb{R}.
  \end{equation}
  Thus, the properties of periodic central curves in $\pi(S^*)$ essentially follows from a simple analysis of one-dimensional dynamics.
\end{remark}

The key ingredient for the invariance \eqref{eq.c-invariance} lies in the fact that the backward iterates $P_{x,t-nT}(\sigma^*_{\hat{x}})$, $P_{x,-nT}(\sigma^*_{\phi_t(\hat{x})})$ have uniformly bounded lengths and they can be joined by the local stable manifolds. See \cite[Lemma 5.3]{XZ} for the details. This idea can be used to show the following:
let $\hat{x}\in\hat{K}$ such that $x=\pi(\hat{x})$ is a periodic point and assume that there is a central segment $I_{\hat{x}}=\{\hat{x}\}\times[0,a]$ satisfying $\hat{f}^n(I_{\hat{x}})\subset\hat{K}\times[0,1]$ for all $n\in\mathbb{Z}$, then it holds
\begin{equation}\label{eq.c-invariance-1}
  P_{x,t}(\gamma_{\hat{x}})\subset\pi(\phi_t(\hat{x}),[0,+\infty)),\quad \forall t\in\mathbb{R},
\end{equation}
where $\gamma_{\hat{x}}=\pi(I_{\hat{x}})$. We then have the following generalization of Lemma \ref{lem.periodic-c-curve}.

\begin{lemma}\label{lem.periodic-c-curve-general}
Suppose all periodic points in $\pi(\hat{K}\times[0,1])$ are hyperbolic. Let $\hat{x}\in\hat{K}$ such that $x=\pi(\hat{x})$ is a periodic point. Assume that there is a central segment $I_{\hat{x}}=\{\hat{x}\}\times[0,a]$ satisfying $\hat{f}^n(I_{\hat{x}})\subset\hat{K}\times[0,1]$ for all $n\in\mathbb{Z}$. 
Then on $\gamma_{\hat{x}}:=\pi(I_{\hat{x}})$ there exists a finite alternate arrangement of hyperbolic saddles and sinks $p_0=x,p_1,\ldots,p_n$ with the following properties.
  \begin{itemize}
    \item Let $(p_k,p_{k+1})$ be the connected component of $\gamma_{\hat{x}}\setminus\{p_0,p_1,\ldots,p_n\}$ with endpoints $p_k$ and $p_{k+1}$ ($0\leq k\leq n-1$), then $(p_k,p_{k+1})$ is contained in the stable set of one endpoint (sink) and in the unstable set of the other endpoint (saddle).
    \item Let $\gamma^*_{\hat{x}}\subset\gamma_{\hat{x}}$ be the subinterval with endpoints $p_0=x$ and $p_n$. If $p_n$ is a sink, then $\gamma_{\hat{x}}\setminus\gamma^*_{\hat{x}}$ (if not empty) is contained in the stable set of $\mathrm{Orb}(p_n)$; otherwise, if $p_n$ is a saddle, then $\gamma_{\hat{x}}\setminus\gamma^*_{\hat{x}}$ is contained in the unstable set of $\mathrm{Orb}(p_n)$. Moreover, in the latter case, suppose $\gamma_{\hat{x}}\setminus\gamma^*_{\hat{x}}$ is not empty, then there exists a periodic sink contained in the unstable set of $\mathrm{Orb}(p_n)$.
  \end{itemize}
\end{lemma}
\begin{proof}
  The result follows from \eqref{eq.c-invariance-1} and a simple analysis of one-dimensional dynamics. For the final assertion, note that when $p_n$ is a saddle and $\gamma_{\hat{x}}\setminus\gamma^*_{\hat{x}}\neq\emptyset$, the forward iterates of any point in $\gamma_{\hat{x}}\setminus\gamma^*_{\hat{x}}$ must approximate to a periodic sink since all the iterates of the central curve $\gamma_{\hat{x}}$ are contained in $\pi(\hat{K}\times[0,1])$ and hence have a uniform length.
\end{proof}

\subsection{Proof of Theorem \ref{thm.codim-one}}
\label{sect.pf-thm-codim-one}

Let $\cG$ be a residual subset of $\xX^1(M,\Gamma)$ such that each vector field in $\cG$ is $\Gamma$-avoiding Kupka-Smale and satisfies Theorem \ref{thm.genericity} and Lemma \ref{lem.sinks-accumulation}.
Let $X\in\cG$ and $\Lambda$ be a compact chain-transitive set of $X$ such that $\Lambda\cap(\Gamma\cup\Sing(X))=\emptyset$. Assume there exists a codimension one partially hyperbolic splitting of the normal bundle $\cN_{\Lambda}=\Delta^s\oplus \Delta^c$.

To prove Theorem \ref{thm.codim-one}, let us assume that $\Lambda$ is not hyperbolic.
By Zorn's lemma, we may assume without loss of generality that $\Lambda$ is minimally nonhyperbolic. Then Lemma \ref{lem.sinks-accumulation} shows that $\Lambda$ is the Hausdorff limit of periodic sinks.

We can extend the splitting $\cN_{\Lambda}=\Delta^s\oplus \Delta^c$ to the maximal invariant set $\Lambda_0$ in a neighborhood $U_0$ of $\Lambda$, such that $U_0\cap(\Gamma\cup\Sing(X))=\emptyset$ and $\cN_{\Lambda_0}=\Delta^s\oplus \Delta^c$ remains partially hyperbolic. Let $T>0$ be a domination constant for the splitting $\cN_{\Lambda_0}=\Delta^s\oplus\Delta^c$, and $P_T$ the map defined on the family of normal manifolds $\{N_x(\delta_0)\}_{x\in \Lambda_0}$ (with some $\delta_0>0$) such that $P_T|_{N_x(\delta_0)}=P_{x,T}$.

Let us take a neighborhood $U$ of $\Lambda$ such that $\overline{U}\subset U_0$ and denote by $K$ the maximal invariant set in $\overline{U}$. Assume first that the center bundle $\Delta^c$ over $K$ has an orientation preserved by $P_T$.
By Theorem \ref{thm.central-model}, one can obtain two central models $(\hat{K}_{-},\hat{f}_-;\pi_-)$ and $(\hat{K}_{+},\hat{f}_+;\pi_+)$ for $(K,P_T)$, which will be called the {\em left} central model and the {\em right} central model, respectively.
Correspondingly, at any point $x\in K$, there are a left central curve $\pi_-(\hat{x}^-,[0,+\infty))$ and a right central curve $\pi_+(\hat{x}^+,[0,+\infty))$. 

By considering the subset of the base $\hat{K}_{+}$ that projects by $\pi_+$ to $\Lambda$, one obtains a central model $(\hat{\Lambda}_{+},\hat{f}_+;\pi_+)$ for $(\Lambda,P_T)$. Similarly, one can also obtain a central model $(\hat{\Lambda}_{-},\hat{f}_-;\pi_-)$ for $(\Lambda,P_T)$.
By Theorem \ref{thm.central-model}, since $\Lambda$ is chain-transitive, the central models $(\hat{\Lambda}_{\iota},\hat{f}_{\iota};\pi_{\iota})$ ($\iota\in\{+,-\}$) for $(\Lambda,P_T)$ are both chain-transitive.
Then according to Lemma \ref{lem.central-model-dichotomy}, we have the following subcases regarding the two central models $(\hat{\Lambda}_{\iota},\hat{f}_{\iota};\pi_{\iota})$ ($\iota\in\{+,-\}$):
\begin{enumerate}[(1)]
\item One of the central models has a chain-recurrent central segment;
\item One of the central models has arbitrarily small trapping strips;
\item Both of the central models have arbitrarily small expanding strips.
\end{enumerate}

In the non-orientable case, by Theorem \ref{thm.central-model} we obtain only one central model $(\hat{K}, \hat{f};\pi)$ for $(K, P_T)$. For this central model, any $x\in K$ is contained in the interior of a central curve which is the union of the two half central curves $\pi(\hat{x}^{\iota},[0,+\infty))$, $\iota=+,-$.
Restricted to the subset $\Lambda$, we obtain a central model $(\hat{\Lambda},\hat{f};\pi)$ with chain-transitive base.

Since the central models $(\hat{K}_{-},\hat{f}_-;\pi_-)$ and $(\hat{K}_{+},\hat{f}_+;\pi_+)$ (or in the non-orientable case, $(\hat{K}, \hat{f};\pi)$) can be taken arbitrarily small (see Remark \ref{rmk.central-model-1}), we can assume that the central curves of the models are all contained in $U_0$, and in particular, any chain-recurrent central segment is contained in $\Lambda_0$.
Our proof shall focus on the orientable case, while the non-orientable case can be treated similarly.

\subsubsection{Local strong stable manifolds and orientation}

By Lemma \ref{lem.partial-hyperbolicity}, there is a partially hyperbolic splitting $T_{\Lambda_0}M=E^s\oplus E^{cu}$ for the tangent flow, such that $E^s\oplus \langle X\rangle=\Delta^s\oplus \langle X\rangle$. It is known that there exist local strong stable manifolds $W^{ss}_{loc}(x)$ tangent to $E^s(x)$ for all $x\in \Lambda_0$, with a uniform size. More generally, let us denote
\[\Lambda_0^+=\bigcap_{t\geq 0}\phi_{-t}(\overline{U_0}).\]
Then there exist local strong stable manifolds with uniform size for all points in $\Lambda_0^+$.
For any $x\in \Lambda_0^+$ and $\vep>0$ small, let $W^{ss}_{\vep}(x)$ be the local strong stable manifold at $x$. For $\vep>0$ small, denote
\begin{equation}\label{eq.ss-mfld}
  D_{\vep}(x)=\phi_{(-\vep,\vep)}(W^{ss}_{\vep}(x)).
\end{equation}
Then $D_{\vep}(x)$ is a codimension-one, $C^1$ embedded submanifold, tangent to $E^s\oplus \langle X\rangle$ on the orbit segment  $\phi_{(-\vep,\vep)}(x)$. 

In the orientable case, attached to each $x\in K$ there are the left central curve $\pi_-(\hat{x}^-,[0,1])$ and the right central curve $\pi_+(\hat{x}^+,[0,1])$.
Now, for any small ball at $x$, it is cut into two connected components by $D_{\vep}(x)$. When the radius is small enough, each component intersects either the left central curve or the right central curve. Then any point in the component which intersects the left (resp. right) central curve is said to be at the left (resp. right) side of $x$. We say that a sequence of points $z_n$ converges to $x$ from the right side if $z_n$ is at the right side of $x$ for $n$ large enough.

Note that in the non-orientable case, we do have a local orientation at each $x\in K$: for a small ball at $x$, it is cut by $D_{\vep}(x)$ into two connected components; we just take one to be the left and the other to be the right.
Then the flow transports this orientation to a local orientation at each point on the orbit of $x$.

\subsubsection{The case of chain-recurrent central segments}

\begin{lemma}\label{lem.recurrent-segments}
  In the orientable subcase (1), there exist periodic orbits in the chain-recurrence class containing $\Lambda$. Moreover, there are periodic sinks contained in the closure of the unstable set of these periodic orbits.
\end{lemma}
\begin{proof}
  Without loss of generality, suppose there exists a chain-recurrent central segment $I=\{\hat{x}^{+}\}\times[0,a_0]$ with $\hat{x}^+\in \hat{\Lambda}_+$ and $\pi_+(\hat{x}^+)=x$.
  Let $\gamma=\pi_+(I)$. As $I$ is a chain-recurrent central segment, one can see that the union $\gamma\cup\Lambda$ is contained in a chain-transitive set $\Lambda'$ of $X$. 
  By Theorem \ref{thm.genericity}, $\Lambda'$ is the Hausdorff limit of some periodic orbits. In particular, there exist periodic points $p_n$ arbitrarily close to the middle point $\pi_+(\hat{x}^+,a_0/2)$ of $\gamma$. Thus $p_n\in D_{\vep}(\gamma)$ for some $\vep>0$, which implies immediately that $p_n$ and $\gamma\cup\Lambda$ are contained in the same chain-recurrence class. This proves the first part of the lemma.

  By definition of chain-recurrent central segment, all iterates of $I$ are well-defined. Let us denote $\gamma_t=P_{x,t}(\gamma)$ for any $t\geq 0$, which is a central curve at $x_t:=\phi_t(x)$. 
  For $\vep>0$ small and for any $t\geq 0$, we denote
  $D_{\vep}(\gamma_t)=\bigcup_{z\in\gamma_t}D_{\vep}(z)$. 
  We may assume that $I$ is contained in the maximal invariant set of $\hat{\Lambda}_+\times[0,a]$ with $a\in(0,1)$ arbitrarily small, see Remark \ref{rmk.dichotomy-central-model}. In other words, the length of $\gamma_t$ is uniformly small. Note that the left central curves $\pi_-(\hat{y}_-,[0,1])$ with $\hat{y}_-\in\hat{K}_-$ have a uniform length.
  It follows that for some $\vep>0$, for any $t\geq 0$, if $y\in D_{\vep}(\gamma_t)\cap K$, then the left central curve $\pi_-(\hat{y}_-,[0,1])$ at $y$ intersects transversely with $D_{2\vep}(x_t)$ at an interior point of $\pi_-(\hat{y}_-,[0,1])$.

  Note that $\Lambda$ is minimally nonhyperbolic, Lemma \ref{lem.sinks-accumulation} shows that $\Lambda$ is the Hausdorff limit of some periodic sinks $\{Q_m\}$. Since $\gamma$ is chain-recurrent and is on the right side of $x$, these periodic sinks cannot accumulate $x$ on the right side. Otherwise, there exists periodic sinks whose local stable manifold intersects $\gamma$, leading to a contradiction. Therefore, the periodic sinks $Q_m$ accumulate $x$ on the left side. The same argument shows that there are periodic sinks accumulating $x_t$ on the left side for each $t\geq 0$.

  Let us take a periodic point $p\in K$, in the same chain recurrence class of $\Lambda$ and arbitrarily close to an interior point of $\gamma$. By partial hyperbolicity and since $\gamma$ is chain-recurrent, $p$ is a saddle with one dimensional unstable manifold. Observe that on the left central curve $\pi_-(\hat{p}_,[0,1])$ at $p$, there exists a neighborhood $V$ of $p$ contained in the unstable set of $p$. If $P_{p,nT}(V)\subset \pi_-(\phi_{nT}(\hat{p}_-),[0,1])$ for all $n\geq 0$, then Lemma \ref{lem.periodic-c-curve-general} shows that the closure of the unstable set of $\mathrm{Orb}(p)$ contains periodic sinks. This gives the second part of the conclusion.

  \begin{figure}
  \centering
    \includegraphics[width=.6\textwidth]{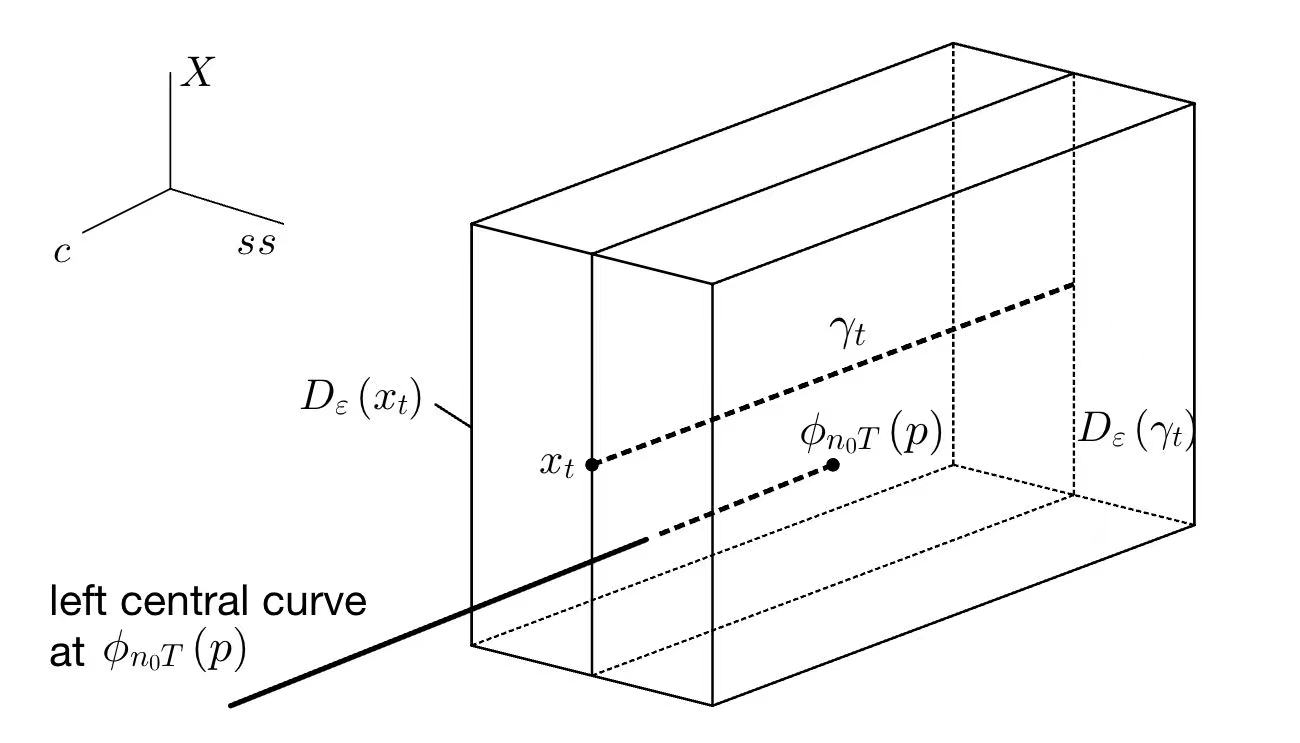}
    \caption{The point $\phi_{n_0T}(p)$ is close to $\gamma_t$ so that the much longer left central curve at $\phi_{n_0T}(p)$ reaches out to the left side of $x_t$.}\label{fig.recurrent-segment}
  \end{figure}

  Suppose now there exists $n>0$ such that $P_{p,nT}(V)\not\subset  \pi_-(\phi_{nT}(\hat{p}_-),[0,1])$. Let $n=n_0$ be the first positive integer with such a property.
  As $p$ is arbitrarily close to $\gamma$, the local stable manifold $W^s_{loc}(\mathrm{Orb}(p))$ intersects $\gamma$. 
  Note that $n_0$ can be assumed to be arbitrarily large by reducing $V$,
  we have $\phi_{n_0T}(p)\in D_{\vep}(\gamma_t)$ for some $t>0$. 
  It follows that the left central curve $\pi_-(\phi_{n_0T}(\hat{p}_-),[0,1])$ at $\phi_{n_0T}(p)$ intersects transversely $D_{2\vep}(x_t)$ and reaches out to the left side of $x_t$. See Figure \ref{fig.recurrent-segment}.
  As there are periodic sinks accumulating $x_t$ on the left side, one finds that $\pi_-(\phi_{n_0T}(\hat{p}_-),[0,1])$ intersects transversely the local strong stable manifolds of the periodic sinks. Observe that $\pi_-(\phi_{n_0T}(\hat{p}_-),[0,1])$ is contained in $P_{p,n_0T}(V)$ and hence contained in the unstable set of $\mathrm{Orb}(p)$. Therefore, there are periodic sinks contained in the closure of the unstable set of $\mathrm{Orb}(p)$.
  This gives the second part of the lemma and completes the proof. 
\end{proof}

\subsubsection{The case of trapping strips}
We show that none of the central models has arbitrarily small trapping strips.

\begin{lemma}\label{lem.trapping-strips}
  In the orientable case, the subcase (2) cannot happen.
\end{lemma}
\begin{proof}
  We prove the lemma by contradiction. Without loss of generality, let us assume $(\hat{\Lambda}_{+},\hat{f}_+;\pi_+)$ has a trapping strip. By continuity and reducing $U$ if necessary, $(\hat{K}_+,\hat{f}_+;\pi_+)$ also has a trapping strip $S$, which projects to the manifold as a continuous family of central curves $\{\gamma_x^+\}_{x\in K}$. Since $S$ is trapping, one has that $\overline{P_{x,T}(\gamma_x^+)}\subset \gamma_{x_T}^+$ for all $x\in K$, where $x_T=\phi_T(x)$. In particular, $\gamma_x^+\subset \Lambda_0^+$ and at each $y\in\gamma_x^+$ there is a $C^1$ embedded disk $D_{\vep}(y)$ as defined by \eqref{eq.ss-mfld}.

  For any periodic point $x\in K$, let $\Gamma^+_x=[x,p_n]$ be the periodic segment on $\gamma^+_x$ given by Lemma \ref{lem.periodic-c-curve}. Let $\eta_x=\gamma^+_x\setminus\Gamma^+_x$. It follows from the trapping property that the endpoint $p_n$ of $\Gamma^+_x$ is a sink and $\eta_x$ is contained in the stable set of $p_n$.
  Note that the trapping strip can be taken arbitrarily small. Lemma \ref{lem.tau-lip} implies that for some constant $C>1$, for any periodic point $x\in K$, if there is another period point $y\in\gamma^+_x$, then it holds
  \begin{equation}\label{eq.trap-strip-1}
    C^{-1}\leq \tau_y/\tau_x\leq C,
  \end{equation}
  where $\tau_x,\tau_y$ are the periods of $x,y$, respectively.

  By Lemma \ref{lem.sinks-accumulation}, we can take a sequence of periodic sinks $q_n$ in $K$ that converges to a point $x_0\in \Lambda$. Let $\tau_n$ be the period of $q_n$ for each $n$.
  Since $\Lambda$ is not a periodic orbit, we can assume that the periods satisfy
  \begin{equation}\label{eq.trap-strip-2}
    \tau_{n+1}/\tau_n>C,\quad\forall n\geq 0.
  \end{equation}
  Let us denote
  \[D_{\vep}(\gamma^+_{q_n}):=\bigcup_{x\in\gamma^+_{q_n}}D_{\vep}(x).\]
  These disks $D_{\vep}(\gamma_{q_n}^+)$ have a uniform size, since the central curves $\gamma_{q_n}^+$ have a uniform length.
  As $q_n$ converges to $x_0$, one has $D_{\vep}(\gamma^+_{q_n})\cap D_{\vep}(\gamma^+_{q_{n+1}})\neq\emptyset$ for $n$ large enough. Since the disks $D_{\vep}(\gamma_{q_n}^+)$ are foliated by local stable manifolds, one can see that either $\mathrm{Orb}(q_n)$ intersects $\Gamma_{q_{n+1}}^+$, or $\mathrm{Orb}(q_{n+1})$ intersects $\Gamma_{q_n}^+$. In both cases, the inequality \eqref{eq.trap-strip-1} shows that $C^{-1}\leq \tau_{n+1}/\tau_n\leq C$, which is a contradiction to \eqref{eq.trap-strip-2}. This completes the proof.
\end{proof}

\subsubsection{Expanding strips on both sides}

\begin{lemma}\label{lem.expanding-both-sides}
  In the orientable subcase (3), $\Lambda$ is contained in a homoclinic class $H$. Moreover, there are periodic sinks contained in the closure of the unstable set of $H$.
\end{lemma}
\begin{proof}
In this case, both central models $(\hat{\Lambda}_{\iota},\hat{f}_{\iota};\pi_{\iota})$ ($\iota=+,-$) have arbitrarily small expanding strips. By continuity and reducing $U$ if necessary, this holds also for the central models $(\hat{K}_{\iota},\hat{f}_{\iota};\pi_{\iota})$. We take for each of the central models a small expanding strip and project them to $M$, obtaining two continuous families of central curves $\{\gamma^+_x\}_{x\in K}$ and $\{\gamma^-_x\}_{x\in K}$.
As in the proof of Lemma \ref{lem.trapping-strips}, we can assume that the central models are small such that for some constant $C>1$, for any periodic point $x\in K$, if $y$ is another periodic point in $\gamma_x^+$ or $\gamma_x^-$, then
\begin{equation}\label{eq.exp-1}
  C^{-1}\leq \tau_y/\tau_x\leq C,
\end{equation}
where $\tau_x, \tau_y$ are the periods of $x,y$, respectively.

By Lemma \ref{lem.sinks-accumulation}, there is a sequence of periodic sinks $O_n\subset K$ converging to $\Lambda$ in the Hausdorff topology. Denote by $\tau_n$ the period of $O_n$. Since $\Lambda$ is not a periodic orbit, we can assume that
\begin{equation}\label{eq.exp-2}
  \tau_{n+1}/\tau_n>C^2,\quad\forall n\geq 0.
\end{equation}
Let us take $p_n\in O_n$ such that $p_n$ converges to a point $x_0\in \Lambda$. Let $\Gamma_n^{\iota}$ be the periodic segment on $\gamma_{p_n}^{\iota}$ given by Lemma \ref{lem.periodic-c-curve}. Precisely,
\[\Gamma_n^{\iota}=\bigcap_{i=0}^{\infty} P_{\phi_{iT}(p_n),-iT}\left(\gamma^{\iota}_{\phi_{iT}(p_n)}\right),\quad \text{for}\ \iota\in\{+,-\}.\]
We define $\Gamma_n=\Gamma^+_n\cup\Gamma^-_n$. Note that $\Gamma_n\subset \Lambda_0^+$. For each $n$, let
  $D_{\vep}(\Gamma_n):=\bigcup_{z\in\Gamma_n}D_{\vep}(z)$, where $D_{\vep}(z)$ is defined as in \eqref{eq.ss-mfld} with $\vep>0$ small.
\begin{claim}
  One has $D_{\vep}(\Gamma_m)\cap D_{\vep}(\Gamma_n)=\emptyset$ for any distinct $m,n\in\mathbb{N}$. Consequently, the length of $\Gamma_n$ goes to 0 as $n\to\infty$.
\end{claim}
\begin{proof}[Proof of claim]
  Suppose the claim is not true, i.e. $D_{\vep}(\Gamma_n)\cap D_{\vep}(\Gamma_m)\neq\emptyset$ for some $m\neq n$. This implies that there is a periodic orbit intersecting both $\Gamma_m$ and $\Gamma_n$. It follows from the inequality \eqref{eq.exp-1} that $C^{-2}\leq \tau_{m}/\tau_{n}\leq C^2$, which is a contradiction to \eqref{eq.exp-2}. This proves the first part of the claim.
  Consequently, one must have $\mathrm{len}(\Gamma_n)\to0$ ($n\to\infty$), since $p_n$ converges to $x_0$. 
\end{proof}

By Lemma \ref{lem.periodic-c-curve}, the two endpoints of the interval $\Gamma_n$ are both saddles. The right endpoint (in $\Gamma^+_n$) will be denoted as $p_n^+$ and  the left endpoint (in $\Gamma^-_n$) as $p_n^-$.
For $\iota\in\{+,-\}$, we denote $J^{\iota}_n=\gamma^{\iota}_{p_n}\setminus\Gamma^{\iota}_n$, which is contained in the unstable set of the saddle $p_n^{\iota}$. By the claim, we can assume that $J^{\iota}_n$ has a uniform length. Let us fix $\vep>0$ small.
Since $N^c$ and $N^{s}$ are transverse, there exists $0<\delta<\vep/2$ such that for any $y\in  K$ and $n$, if $d(p_n^+,y)<\delta$ and if $y$ is on the right of $p^+_n$, then $D_{\vep}(y)\pitchfork J^+_n\neq\emptyset$. A similar result holds on the left side.

\begin{figure}[htbp]
  \centering
  \includegraphics[width=.5\textwidth]{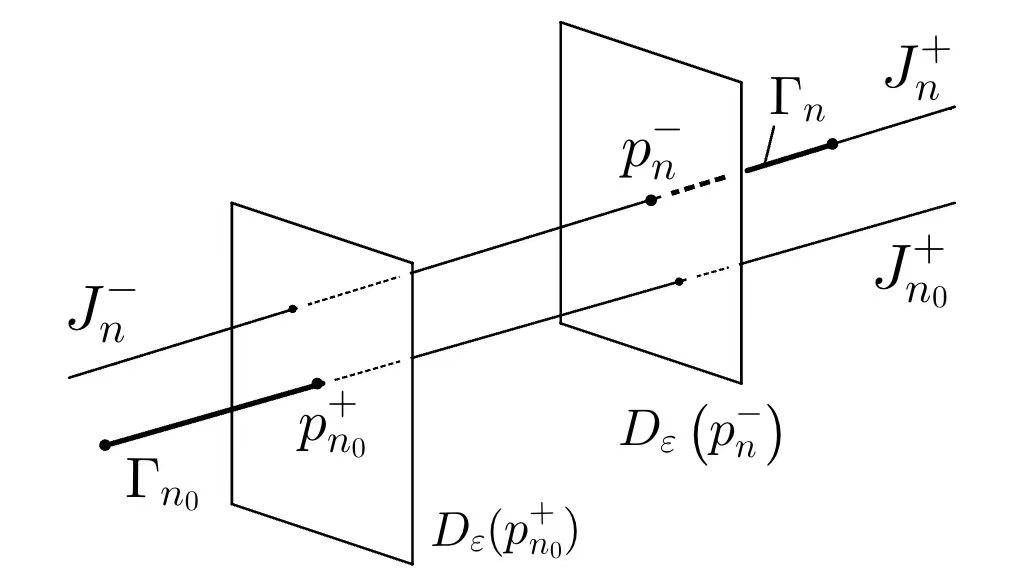}
  \caption{The two periodic orbits $\mathrm{Orb}(p^+_{n_0})$ and $\mathrm{Orb}(p^-_n)$ are homoclinically related}
  \label{fig.expanding-strip}
\end{figure}

Now, we can take $n_0$ large enough, such that for any $n>n_0$, the distance between $p_n$ and $p_{n_0}$ is smaller than $\delta/4$. Moreover, the length of $\Gamma_n$ is also smaller than $\delta/4$.
By the claim, $D_{\vep}(\Gamma_m)\cap D_{\vep}(\Gamma_n)=\emptyset$ for any $m\neq n$.
It follows that the local flow around $\Gamma_{n_0}$ gives an ordering of the intervals $\Gamma_n$.
Without loss of generality, we assume that $\Gamma_n$ is on the right of $\Gamma_{n_0}$. Then our choice of the constants implies that $D_{\vep}(p^+_{n_0})\cap J_n^-\neq\emptyset$ and $D_{\vep}(p^-_{n})\cap J_{n_0}^+\neq\emptyset$, i.e. the two periodic orbits $\mathrm{Orb}(p^+_{n_0})$ and $\mathrm{Orb}(p^-_n)$ are homoclinically related. See Figure \ref{fig.expanding-strip}.
As the length of $\Gamma_n$ goes to zero and $p_n\to x_0\in\Lambda$, we see that $\Lambda$ is contained in $H(p^+_{n_0})$.

Since $\Gamma_n$ contains the periodic sink $p_n$, one has $\mathrm{len}(\Gamma_{n}^{\iota})>0$ for $\iota\in\{+,-\}$ and for all $n$. It follows that there exists a sink on $\Gamma_{n_0}^+$, in the closure of the unstable set of $p^+_{n_0}$, see Lemma \ref{lem.periodic-c-curve}. 
This completes the proof of the lemma.
\end{proof}

\subsubsection{Completing the proof of Theorem \ref{thm.codim-one}}

Considering the central model(s) for $\Lambda$.
In the orientable case, the conclusion of Theorem \ref{thm.codim-one} follows from Lemma \ref{lem.recurrent-segments} $\sim$ \ref{lem.expanding-both-sides}. In the non-orientable case, Lemma \ref{lem.central-model-dichotomy} implies that either the central model $(\hat{\Lambda},\hat{f};\pi)$ has a chain-recurrent central segment, similar to subcase (1); or it has arbitrarily small trapping  strips for $\hat{f}$, similar to subcase (2); or it has arbitrarily small expanding strips for $\hat{f}$, similar to subcase (3). Hence the conclusion of Theorem \ref{thm.codim-one} follows from similar arguments as in the orientable subcases.

\section{Lyapunov stable chain recurrence classes}\label{sect.pf-of-h-class}

This section is devoted to the proof of Theorem \ref{thm.h-class}\,: let $X\in\xX^1(M)$ be a $C^1$ generic vector field and $C(\sigma)$ a non-trivial chain recurrence class of a singularity $\sigma$, if  all singularities in $C(\sigma)$ are Lorenz-like and there exists a partially hyperbolic splitting $T_{C(\sigma)}M=E^s\oplus F$ with $E^s$ contracting and $\dim F=2$, then $C(\sigma)$ is a homoclinic class.

To prove Theorem \ref{thm.h-class}, i.e. to show that $C(\sigma)$ is a homoclinic class, we follow the strategy of \cite[Theorem 4.4]{GY}. In fact, we only provide a sketch of the proof, as one can refer to \cite[Section 6]{GY} for the details.

Observe first that the assumptions of Theorem \ref{thm.h-class} imply that the unstable manifold of each singularity is one-dimensional and $C(\sigma)$ is Lyapunov stable (see Remark \ref{rmk.h-class}).

Proving by contradiction, suppose that $C(\sigma)$ is not a homoclinic class, or equivalently in the generic case, $C(\sigma)$ is aperiodic.
The following result shows that $C(\sigma)$ cannot be singular hyperbolic.
\begin{theorem}[\cite{PYY}]\label{thm.yang}
Let $C$ be a non-trivial Lyapunov stable chain recurrence class of a $C^1$ vector field. Suppose $C$ is singular hyperbolic and every singularity in it is hyperbolic, then $C$ contains a periodic orbit. Moreover, if the vector field is $C^1$ generic, then $C$ is an attractor.
\end{theorem}

As $C(\sigma)$ is not singular hyperbolic, Zorn's lemma suggests a definition as the following:
\begin{definition}
  A compact invariant set is called an {\em $\cN$-set}, if the set itself is not singular hyperbolic, but any compact invariant proper subset is singular hyperbolic.
\end{definition}

The rest of the proof can be divided into 3 parts:
\begin{enumerate}
  \item (Subsection \ref{sect.cross-section-system}) Cross-section system: construction and properties;
  \item (Subsection \ref{sect.N-sets}) Study of $\cN$-sets: existence and properties;
  \item (Subsection \ref{sect.adapted-returns}) Infinitesimal $cu$ adapted returns via homoclinic orbits of singularities: the contradiction.
\end{enumerate}

\subsection{Construction of a cross-section system}\label{sect.cross-section-system}

Let $C(\si)\cap\Sing(X)=\{\si_1,\cdots,\si_k\}$. Since the unstable manifold of each $\si_i$ is one-dimensional, we can fix an orientation around $\si_i$, and let $W^u(\si_i)\setminus\{\si_i\}$ be composed of the left separatrix $W^{u,l}(\si_i)$ and the right separatrix $W^{u,r}(\si_i)$. Moreover, since each $\si_i$ is Lorenz-like, one has $E^s(\si_i)=E^{ss}(\si_i)\oplus E^{c}(\si_i)$, $\dim E^{c}(\si_i)=1$. It follows that the orientation we have fixed also allows us to give a positive direction of $E^c(\si_i)$ for each $\si_i$, which is equivalent to fixing a positive side of $E^{ss}(\si_i)\oplus E^u(\si_i)$.

Let $S\subset M$ be a $C^1$ hypersurface which is homeomorphic to $(-1,1)^{d-1}$, where $d=\dim M$. We say that $S$ is a {\em cross-section} of $X$ if $\measuredangle(T_xS,\langle X(x)\rangle)>\pi/4$ for each $x\in S$.

For each $\sigma_i$, one can take two cross-sections $S^+_i$, $S^-_i$ in a small neighborhood of $\si_i$ such that they are on the opposite sides of $E^{ss}(\si_i)\oplus E^u(\si_i)$. One requires that the following conditions are satisfied for $\pm\in\{+,-\}$:
\begin{enumerate}
  \item There exists a homeomorphism $h^{\pm}_i:[-1,1]\times[-1,1]^{d-2}$ $\to \Cl(S^{\pm}_i)$, such that $h^{\pm}_i((-1,1)\times(-1,1)^{d-2})=S^{\pm}_i$.
  \item There is a uniform $\alpha>0$ such that for each $x\in[-1,1]$, the image of $h(x,\cdot)$ is $W^{ss}_{loc}(\phi_{[-\alpha,\alpha]}(h(x,0)))\cap \overline{S}$.
  \item Moreover, $S^{\pm}_i\cap W^s_{loc}(\si_i)=h^{\pm}_i(\{0\}\times (-1,1)^{d-2})$.
\end{enumerate}
Denote by $l^{\pm}_i=S^{\pm}_i\cap W^s_{loc}(\sigma_i)$. Then  $S^{\pm}_i\setminus l^{\pm}_i=S^{\pm,l}_i\cup S^{\pm,r}_i$, i.e.  $l^{\pm}_i$ cuts $S^{\pm}_i$ into two parts: the left part $S^{\pm,l}_i$ and the right part $S^{\pm,r}_i$.

One defines the union of cross-sections associated to the singularities $\{\sigma_i\}$:
\[\Sigma=\bigcup_{1\leq i\leq k}(S^+_i\cup S^-_i).\]
For each $p\in\Sigma$, if $\phi_t(p)\in\Sigma$ for some $t>0$, one can define the first return time $t_p=\min\{t>0:\phi_t(p)\in\Sigma\}$ and define $R(p)=\phi_{t_p}(p)$. The map $R$ will be called the {\em first return map} associated to $\Sigma$. Let $\mathrm{Dom}(R)$ be the domain of $R$, one has
\[\mathrm{Dom}(R)=\{p\in\Sigma:\exists t>0 \ \text{s.t.}\ \phi_t(p)\in\Sigma\}.\]

\begin{definition}[Cross-section system]
\label{def.cross-section-system}
With the notations above, we call $(\Sigma,R)$ a cross-section system of $(C(\sigma),X)$ if the following conditions are satisfied:
\begin{enumerate}
\item $\partial\Si\cap C(\sigma)=\emptyset$, where $\partial\Si$ is the boundary of $\Sigma$.
\item For each $p\in\Sigma$, suppose $p\in S^{\pm}_i$, then there exists $\vep>0$ such that $W^{ss}_{loc}(\phi_{(-\vep,\vep)}(p))\cap\Sigma\subset h^{\pm}_i(\{x_p\}\times(-1,1)^{d-2})$, where $h^{\pm}_i(x_p,y_p)=p$.
\item For each $\si_i$, and $p\in W^s_{loc}(\si_i)\cap\La\setminus\{\si_i\}$, there exists $t\in\mathbb{R}$ such that $\phi_t(p)\in l^{\pm}_i$, where $l^{\pm}_i=S^{\pm}_i\cap W^s_{loc}$. Moreover, if $t>0$, then $\phi_{[0,t]}(p)\subset W^s_{loc}(\si_i)$, and if $t<0$, then $\phi_{[t,0]}(p)\subset W^s_{loc}(\si_i)$. 
\item There is $\alpha_0\in(0,1)$ such that for any $p\in \mathrm{Dom}(R)$, suppose $p=h^{\pm}_i((x_p,y_p))$ and $q=R(p)=h^{\pm}_j((x_q,y_q))$, then $h^{\pm}_i((x_p,t))\in \mathrm{Dom}(R)$ for any $t\in[-1,1]^{d-2}$ and $R(h^{\pm}_i(\{x_p\}\times[-1,1]^{d-2})\subset h^{\pm}_j(\{x_q\}\times(-\alpha_0,\alpha_0)^{d-2})$.
\item For any $x\in\La\setminus\bigcup_{1\leq i\leq k}W^s_{loc}(\si_i)$, the positive orbit of $x$ will intersect $\Si$. Especially,
\[C(\sigma)\cap\Si\setminus\bigcup_{1\leq i\leq k}W^s_{loc}(\si_i)\subset\mathrm{Dom}(R).\]
\end{enumerate}
\end{definition}

The existence of a cross-section system for the vector field $X$ is given by Proposition 6.4 in \cite{GY}. Although the construction of a cross-section system in \cite{GY} is for $\dim E^{ss}=1$, it can be generalized to the case $\dim E^{ss}>1$ almost immediately, giving the following result.
\begin{proposition}
  Assuming that $C(\sigma)$ is aperiodic, then $C(\sigma)$ admits a cross-section system $(\Sigma_X,X)$.
\end{proposition}

One defines cross-section systems similarly for any $C^1$ vector field close to $X$.
Since $X$ is $C^1$ generic, there is a $C^1$ neighborhood $\cU_X$ of $X$ such that the following properties hold for each $Y\in\cU_X$:
\begin{enumerate}[(P1)]
  \item Every singularity $\rho\in C(\sigma_Y,Y)$ is Lorenz-like and $W^{ss}(\rho,Y)\cap C(\sigma_Y,Y)=\{\rho\}$. (\cite[Lemma 3.19]{GY})
  \item If $Y$ is weak Kupka-Smale, i.e. all critical elements of $Y$ are hyperbolic, then $C(\sigma_Y,Y)$ is Lyapunov stable. (\cite[Lemma 3.15]{GY}) \label{p.ls} 
  \item A singularity $\rho$ is contained in $C(\sigma_Y,Y)$ if and only if it is the continuation of a singularity in $C(\sigma)$. In particular, the number of singularities in $C(\sigma_Y,Y)$ is a constant $k$ for all $Y\in \cU_X$. (\cite[Lemma 3.12]{GY})
\end{enumerate}

As shown in \cite[Proposition 6.5]{GY}, one can assume that the neighborhood $\cU_X$ is small such that for any $Y\in\cU_X$, $C(\si_Y,Y)$ admits also a cross-section system $(\Si_Y,R_Y)$ such that $\Si_Y$ is close to $\Si_X$ in the sense that $\Si_Y$ is obtained by modifying the boundary of $\Si_X$ slightly.

We choose a smaller $C^1$ neighborhood $\cV_X$ of $X$ such that $\overline{\cV_X}\subset \cU_X$.
For each $Y\in \cV_X$, let us define $n(Y)$ to be the number of homoclinc orbits of singularities in $C(\sigma_Y,Y)$. Here, a homoclinic orbit of a singularity is a non-trivial orbit contained in the intersection of its stable manifold and unstable manifold. Remember that the unstable manifold of each singularity is one-dimensional. Thus, $n(Y)\leq 2k$. Let
\[n_0=\max\{n(Y): Y\in\cV_X, \ Y\ \text{is weak Kupka-Smale}\}.\]
Denote by $\xM_{n_0}$ the set of $Y\in\cV_X$ such that $n(Y)=n_0$.

Let us take $Y_0\in\xM_{n_0}$.
Denote by $\Gamma$ the closure of the union of all the homoclinic orbits of singularities in $C(\sigma_{Y_0},Y_0)$.
We then consider the set of vector fields
\[\xX^1(M,\Gamma)=\{Y\in\xX^1(M): Y|_{\Gamma}=Y_0|_{\Gamma}\}.\]

\subsection{Existence and properties of $\cN$-sets}
\label{sect.N-sets}

\begin{proposition}\label{prop.N-set}
  Suppose $C(\sigma)$ is not singular hyperbolic, then there exists a $C^1$ neighborhood $\cU$ of $X$ such that for any vector field $Y\in \cU$, the continuation of $C(\sigma)$, denoted as $C(\sigma_Y,Y)$ is not singular hyperbolic and contains a transitive $\cN$-set $\Lambda_Y$.
  Moreover, there exists an ergodic measure $\mu_Y$ such that $\mathrm{supp}(\mu_Y)=\Lambda_Y$ and for any $t>0$, one has
  \begin{equation}\label{eq.measure}
    \int \log|\det(\Phi^Y_t|F_Y(x))|d\mu_Y\leq 0.
  \end{equation}
\end{proposition}

\begin{proof}
  Suppose $C(\sigma)$ is not singular hyperbolic. Note that the continuation $C(\sigma_Y,Y)$ is well-defined as long as $Y$ is $C^1$ close enough to $X$, and it varies upper semi-continuously. We claim that there is a $C^1$ neighborhood $\cU$ of $X$ such that for any $Y\in\cU_0$, the continuation $C(\sigma_Y,Y)$ is not singular hyperbolic. Suppose the claim does not hold, i.e. in any $C^1$ neighborhood of $X$ there is a vector field $Y$ such that $C(\sigma_Y,Y)$ is singular hyperbolic. Then since $X$ is generic, the upper semi-continuity of chain recurrence classes would imply that $C(\sigma)$ is already singular hyperbolic, a contradiction. This proves the claim.

  By Zorn's lemma, for any $Y\in\cU$, the chain recurrence class $C(\sigma_Y,Y)$ contains an $\cN$-set $\Lambda_Y$. By \cite[Lemma 6.11]{GY}, $\Lambda_Y$ can be chosen transitive.
  To see the existence of $\mu_Y$, one can refer to \cite[Corollary 6.12]{GY}.
\end{proof}

By the proposition, we may assume that the neighborhood $\cV_X$ is small enough such that for each $Y\in\cV_X$, $C(\sigma_Y,Y)$ is not singular hyperbolic and hence contains a transitive $\cN$-set $\Lambda_Y$.

\begin{lemma}\label{lem.existence-of-singularity}
  Let $Y\in\xX^1(M,\Gamma)\cap \cV_X$ be a generic vector field in $\xX^1(M,\Gamma)$. Suppose $\Lambda_Y$ is any $\cN$-set contained in $C(\sigma_Y,Y)$, then $\Lambda_Y\cap\Sing(Y)\neq\emptyset$. Moreover, $\Lambda_Y$ is not reduced to a singularity.
\end{lemma}
\begin{proof}
  By Theorem \ref{thm.Gamma-KS}, $Y$ is $\Gamma$-avoiding Kupka-Smale. Then as $\Gamma$ consists of a set of singularities and their homoclinic orbits, one can see that $Y$ is weak Kupka-Smale, and moreover, $Y\in\xM_{n_0}$. By property (P\ref{p.ls}), $C(\sigma_Y,Y)$ is Lyapunov stable. Note that singularities in $C(\sigma_Y,Y)$ are Lorenz-like and hence singular hyperbolic, we need only prove the first part of the lemma.
  Proving by contradiction, suppose that $\Lambda_Y\cap\Sing(Y)=\emptyset$, then $\Lambda_Y$ is codimension one partially hyperbolic. Moreover, since $\Gamma$ is a union of singularities and their associated homoclinic orbits, one has $\Lambda_Y\cap\Gamma=\emptyset$. Then by Theorem \ref{thm.codim-one}, there are periodic sinks contained in closure of the unstable set of $C(\sigma_Y,Y)$, a contradiction to the fact that $C(\sigma_Y,Y)$ is Lyapunov stable.
\end{proof}

\begin{remark}
  This lemma accounts for the essential difference (for the proof) between the 3-dimensional case in \cite[Section 6]{GY} and the higher-dimensional case in the present paper. It allows us to get around the $C^2$ argument in \cite{GY}. More discussions can be found in Remark \ref{rmk.codim-one-1}.
\end{remark}

\subsection{Infinitesimal $cu$ adapted returns: the contradiction}
\label{sect.adapted-returns}
The final subsection defines the notation of {\em infinitesimal $cu$ adapted returns} for the cross-section system as defined in Section \ref{sect.cross-section-system}, and uses it to prove Theorem \ref{thm.h-class} by obtaining a contradiction: on the one hand, ones shows that every weak Kupka-Smale vector field in $\cV_X$ has no infinitesimal $cu$ adapted returns; on the other hand, for any $Y\in\xX^1(M,\Gamma)\cap\cV_X$ which is generic in $\xX^1(M,\Gamma)$, it can be proved that there is $Z$ arbitrarily close to $Y$ such that $Z$ is weak Kupka-Smale and has an infinitesimal $cu$ adapted return.

To define the infinitesimal $cu$ adapted returns, one looks at the sets $l^{\pm}_i$ (see Definition \ref{def.cross-section-system}) and the returns of small $cu$-curves that starts from $l^{\pm}_i$.
Roughly speaking, an infinitesimal $cu$ adapted return for the cross-section system $(\Sigma_Y,R_Y)$ corresponds to a periodic chain of returns of arbitrarily small $cu$-curves  on $\Sigma_Y$ that starts from some $l^{\pm}_i$ (and moreover, the periodic chain of returns shall satisfy an extra condition which essentially ensures that the basin of some sink accumulates on $C(\sigma_Y,Y)$). See Definition 6.19 in \cite{GY} for the precise statements.

Observe that the key point in the definition of infinitesimal $cu$ adapted returns is the orientation of $cu$-curves in the cross-sections $S^{\pm}_i$, see Definition 6.17 in \cite{GY}. For us, what matters is that the unstable manifold of each singularity is one-dimensional and hence gives an orientation of left and right.
This is no different from the 3-dimensional case.

Now, the following two propositions (Proposition \ref{prop.no-adapted-return} and \ref{prop.exist-adapted-return}) shall give the contradiction and complete the proof of Theorem \ref{thm.h-class}.
\begin{proposition}[{\cite[Proposition 6.20]{GY}}]
\label{prop.no-adapted-return}
  Every weak Kupka-Smale vector field in $\cV_X$ has no infinitesimal $cu$ adapted returns.
\end{proposition}
The general idea behind the proposition is the following: if for some weak Kupka-Smale $Y\in\cV_X$ there exists an infinitesimal $cu$ adapted return for $(\Sigma_Y,R_Y)$, then one can show that $C(\sigma_Y,Y)$ is accumulated by the basin of some sink, a contradiction to the fact that $C(\sigma_Y,Y)$ is Lyapunov stable.

The next result corresponds to Proposition 6.25 in \cite{GY} and shares the same proof.
\begin{proposition}\label{prop.exist-adapted-return}
  For any $Y\in\xX^1(M,\Gamma)\cap\cV_X$ which is generic in $\xX^1(M,\Gamma)$, there is $Z$ arbitrarily close to $Y$ such that $Z$ is weak Kupka-Smale and has an infinitesimal $cu$ adapted return for the cross-section system $(\Sigma_Z,R_Z)$.
\end{proposition}

For the proof of this result, note that $Y$ has already $n_0$ homoclinic orbits of singularities in $C(\sigma_Y,Y)$, hence one can not obtained an extra homoclinic orbit of singularities by perturbation, without breaking the existing ones. This imposes rather strong restrictions on the $\cN$-set of $\Lambda_Y\subset C(\sigma_Y,Y)$, such as every singularity in $\Lambda_Y$ has at least one homoclinic orbit \cite[Lemma 6.15]{GY}, and if a singularity in $\Lambda_Y$ has only one homoclinic orbit, then the homoclinic orbit must be contained in $\Lambda_Y$ \cite[Lemma 6.16]{GY}. Moreover, one can show that if a singularity in $\Lambda_Y$ has two homoclinic orbits, then one can perturb the system to obtain a weak Kupka-Smale $Z$ such that $(\Sigma_Z,R_Z)$ has an infinitesimal $cu$ adapted return \cite[Lemma 6.24]{GY}. Finally, these results can be used for a proof by contradiction, with a careful study of the return map of the cross-section system.

\vspace{0.5cm}
\noindent {\em Shaobo Gan}, School of Mathematical Sciences, Peking University, Beijing 100871, China\\
E-mail address: gansb@pku.edu.cn

\vspace{0.5cm}
\noindent {\em Ruibin Xi}, School of Mathematical Sciences, Peking University, Beijing 100871, China\\
E-mail address: ruibinxi@math.pku.edu.cn

\vspace{0.5cm}
\noindent {\em Jiagang Yang}, Departamento de Geometria, Instituto de Matem\'{a}tica e
Estat\'{i}stica, Universidade Federal Fluminense, Niter\'{o}i, Brazil\\
E-mail address: yangjg@impa.br
\vspace{0.5cm}

\noindent {\em Rusong Zheng}, Joint Research Center on Computational Mathematics and Control, Shenzhen MSU-BIT University, Shenzhen 518172, China\\
E-mail address: zhengrs@smbu.edu.cn
\end{document}